\documentclass[12pt]{article}

\usepackage{graphicx}
\usepackage{amsmath}
\usepackage{amssymb}
\usepackage{latexsym}
\usepackage{subfigure}
\usepackage{crop}
\usepackage{algorithmic}
\usepackage{algorithm}
\usepackage{multirow}
\usepackage[section,subsection,subsubsection]{placeins}
\usepackage[colorlinks = true, pdfstartview = FitV, linkcolor = blue, citecolor = blue, urlcolor = blue]{hyperref}
\usepackage{enumerate}

\newcommand {\uu}  { {\bf u} }

\newcommand {\xx}  { {\bf x} }

\newcommand{\R}{{\rm I\!R}}

\newcommand\argmin[1]  {\underset{#1}{\operatorname{arg\ min}}}

\newcommand {\Ex} { {\mathbb E} }

\newcommand {\qq}  { {\bf q} }

\newcommand {\mm}  { {\bf m} }
\newcommand {\vv}  { {\bf v} }
\newcommand {\ww}  { {\bf w} }

\newcommand {\FF}  { {\bf F} }

\newcommand {\dd}  { {\bf d} }

\newcommand{\hf}{\frac12}

\renewcommand{\vec}[1]{\ensuremath{\mathbf{#1}}}

\renewcommand{\div}{\nabla\cdot\,}
\newcommand{\grad}{\ensuremath {\vec \nabla}}

\newcommand{\add}[1]{\textcolor{black}{#1}}

\newtheorem{theorem}{Theorem}

\newtheorem{example}{Example}

\newenvironment{proof}[1][Proof]{\begin{trivlist}
\item[\hskip \labelsep {\bfseries #1}]}{\end{trivlist}}

\newcommand{\qed}{\nobreak \ifvmode \relax \else
      \ifdim\lastskip<1.5em \hskip-\lastskip
      \hskip1.5em plus0em minus0.5em \fi \nobreak
      \vrule height0.75em width0.5em depth0.25em\fi}

\begin{document}

\title{Data completion and stochastic algorithms for PDE inversion problems 
with many measurements}

\author{Farbod Roosta-Khorasani, Kees van den Doel and Uri Ascher
\thanks{Dept. of Computer Science, University of British Columbia, Vancouver, Canada
{\tt farbod/kvdoel/ascher@cs.ubc.ca} .
This work was supported in part by NSERC Discovery Grant 84306.}}


\maketitle

\begin{abstract}

Inverse problems involving systems of partial differential equations (PDEs)
with many measurements or experiments
can be very expensive to solve numerically.
In a recent paper we examined 
dimensionality reduction methods, both stochastic
and deterministic, to reduce this computational burden,
assuming that all experiments share the same set of receivers.

In the present article we consider the more general and practically important case where receivers are not
shared across experiments. We propose a data completion approach to alleviate this problem.
This is done by means of an approximation using an appropriately restricted
gradient or Laplacian regularization, extending existing data
for each experiment to the union of all receiver locations. 
Results using the method of simultaneous sources (SS) with the completed data 
are then compared to those obtained by a more general 
but slower random subset (RS) method which requires no modifications.
\end{abstract}

%
%


\section{Introduction}
\label{sec:int}

The reconstruction of distributed parameter functions, by fitting to measured data solution values
of partial differential equation (PDE) systems in which  they appear as material properties,
can be very expensive to carry out.
This is so especially in cases where there are many experiments, where just one evaluation of
the forward operator can involve hundreds and thousands of PDE solves.
And yet, there are several such problems of intense current interest
in which the use of many experiments
is crucial for obtaining credible reconstructions in practical situations
\cite{na,dmr,jima12,haasol,fichtner,hel,rnkkda,learhe,smvoz,pihakn,cin,bbp,doasha}.
Extensive theory (e.g.,~\cite{papauh,aspa,alve,jima12}) also suggests that many well-placed experiments
are often a practical must for obtaining credible reconstructions.
Thus, methods to alleviate the resulting computational burden are highly sought after.


To be more specific,
consider the problem of recovering a model $\mm \in \R^{l_m}$, representing a discretization
of a surface function $m(\xx)$ in 2D or 3D, from measurements
$\dd_i \in \R^{l}$, $i = 1, 2, \ldots , s$.\footnote{For notational
  simplicity we make the nonessential assumption that $l$ does not
  depend on the experiment $i$.} For each $i$, the data is predicted as a function
of $\mm$ by a forward operator $\FF_i$, and the goal is to find (or infer) $\mm = \mm^*$ such that 
the misfit function
\begin{eqnarray} 
\phi(\mm) =  \sum_{i=1}^s \| \FF_i(\mm) - \dd_i \|^2
\label{1.1}
\end{eqnarray}
is roughly at a level commensurate with the noise.\footnote{
Throughout this article we use the $\ell_2$ vector norm unless otherwise
specified.}
The forward operator
involves an approximate solution of a PDE system, which
we write in discretized form as
\begin{subequations}
\begin{eqnarray}
A(\mm) \uu_i = \qq_i, \quad i = 1, \ldots , s ,
\label{1.5a}
\end{eqnarray}
where $\uu_i \in \R^{l_u}$ is the $i$th field, $\qq_i \in \R^{l_u}$ is the $i$th source,
and $A$ is a square matrix discretizing the PDE plus appropriate side conditions. 
Furthermore, there are given projection matrices $P_i$ such that
\begin{eqnarray}
\FF_i(\mm) = P_i\uu_i = P_i A(\mm)^{-1} \qq_i \label{1.5b}
\end{eqnarray}
predicts the $i$th data set.
Thus, evaluating $\FF_i$ requires a PDE system solve, and evaluating the objective function $\phi(\mm)$
requires $s$ PDE system solves.
\label{1.5}
\end{subequations}

For reducing the cost of evaluating \eqref{1.1}, stochastic approximations are natural.
Thus, introducing a random vector $\ww = (w_1, \ldots, w_s)^T$
from a probability distribution satisfying
\begin{eqnarray} 
\mathbb{E} ( \ww \ww^T ) = I 
\label{w_cond}
\end{eqnarray}
(with $\Ex$ denoting the expected value with respect to $\ww$ and $I$ the 
$s \times s$ identity matrix), 
we can write~\eqref{1.1} as
\begin{eqnarray} 
\phi(\mm) = \Ex \left( \|  \sum_{i=1}^s  w_i( \FF_i(\mm) - \dd_i) \|^2 \right) ,
\label{1.4}
\end{eqnarray}
and approximate the expected value by a few samples $\ww$~\cite{achlioptas}.
If, furthermore, 
the data sets in different experiments are measured at the same
locations, i.e., $P_i = P \ \forall i$, then
\begin{eqnarray} 
\sum_{i=1}^s w_i\FF_i = \sum_{i=1}^s w_iP_iA(\mm)^{-1}\qq_i = PA(\mm)^{-1} \big( \sum_{i=1}^s w_i\qq_i\big) ,
\label{1.6}
\end{eqnarray}
which can be computed with a single PDE solve per realization of the weight vector $\ww$,
so a very effective procedure for approximating the objective function $\phi(\mm)$ 
is obtained \cite{HaberChungHermann2010}.

Next, in an iterative process for reducing~\eqref{1.1} sufficiently,
consider approximating the expectation value at iteration $n$ 
by random sampling from a set of $s_n$ vectors $\ww$, with $s_n \leq s$ potentially satisfying $s_n \ll s$;
see, e.g.,~\cite{sdr,jlns,geisser}. Several recent papers have proposed methods to control the size $s_n$
\cite{doas3,frsc11,bcnn,rodoas}. Let us now concentrate on one such iteration $n$, for which a specialized
Gauss-Newton (GN) or L-BFGS method may be employed.
We can write \eqref{1.1} using the Frobenius norm $\| \cdot \|_F$ as
\begin{eqnarray} 
\phi(\mm) &=& \| F(\mm) - D \|_F^2 , \label{1.7} \\
& & F = \big[ \FF_1, \FF_2, \ldots , \FF_s\big] \in \R^{l \times s}, \;\;\; 
D = \big[ \dd_1, \dd_2, \ldots , \dd_s\big] \in \R^{l \times s}, \nonumber
\end{eqnarray}
and hence, an unbiased estimator of $\phi(\mm)$ in the $n$th iteration is
\begin{equation} 
\hat \phi(\mm, W) = \frac{1}{s_n} \| ( F(\mm) - D ) W \|_F^2 ,
\label{phiMW}
\end{equation}
where $W=W_n = \big[ \ww_1, \ww_2, \ldots , \ww_{s_n} \big]$ is an $s \times s_n$ matrix with $\ww_{j}$'s 
drawn from any distribution satisfying~\eqref{w_cond}. 
For the case where $P_i = P \ \forall i$, different methods of simultaneous sources (SS) 
are obtained by using different algorithms for this
{\em model and data reduction} process~\cite{avto,yori,roas}. 
In \cite{rodoas} we have discussed and compared three such methods:
(i) a Hutchinson random sampling, (ii) a Gaussian random sampling, and (iii) the deterministic truncated 
singular value decomposition (TSVD).
We have found that, upon applying these methods to the famous DC-resistivity problem,
their performance was roughly comparable (although for just estimating the misfit function by~\eqref{phiMW},
only the stochastic methods work well).

A fourth, random subset (RS) method was considered in \cite{rodoas,doas3},
where a random subset of the original
experiments is selected at each iteration $n$. 
This method does not require that the receivers be shared among different experiments.
However, its performance was found 
to be generally worse than the methods of simultaneous
sources, roughly by a factor between $1$ and $4$, and on average about $2$.\footnote{
The relative efficiency factor further increases if a less conservative criterion is used
for algorithm termination, see Section~\ref{sec:results}.}
This brings us to the quest of the present article, namely, to seek methods for the general case where
$P_i$ does depend on $i$, which are as efficient as the simultaneous sources methods.
The tool employed for this is to ``fill in missing data'', thus
replacing $P_i$, for each $i$, by a common projection
matrix $P$ to the union of all receiver locations, $i = 1, \ldots , s$.

The prospect of such {\em data completion}, like that of casting a set of false teeth based on a few
genuine ones, is not necessarily appealing, but is often necessary for reasons of computational efficiency.
Moreover, applied mathematicians do a virtual data completion automatically when considering a 
Dirichlet-to-Neumann map, for
instance, because such maps assume knowledge of the field $u$ (see, e.g., \eqref{2.1} below)
or its normal derivative on the entire spatial domain boundary,
or at least on a partial but continuous segment of it. 
Such knowledge of noiseless data at uncountably many locations
is never the case in practice, where receivers are discretely located and some noise, 
including data measurement noise, is unavoidable.
On the other hand, it can be argued that any practical data completion must inherently destroy
some of the ``integrity'' of the statistical modeling underlying, for instance, 
the choice of iteration stopping criterion,
because the resulting ``generated noise'' at the false points is not statistically independent 
of the genuine ones where data was collected.

Indeed, the problem of proper data completion is far from being a trivial one, and its inherent
difficulties are often overlooked by practitioners.
In this article we consider this problem in the context of the DC-resistivity problem 
(Section~\ref{sec:detail}),
with the sources and receivers for each data set located at \add{segments of}
the boundary $\partial \Omega$ of the
domain on which the forward PDE is defined.
Our data completion approach is to approximate or interpolate the given data directly in 
smooth segments of the boundary, while taking advantage of 
prior knowledge as to how the fields $\uu_i$ must behave there.
\add{We emphasize that the sole purpose of our data completion algorithms is to allow the set of 
receivers to be shared among all experiments. 
This can be very different from traditional data completion efforts that have sought to obtain 
extended data throughout the physical domain's boundary or even in the entire physical domain. 
Our ``statistical crime'' with respect to noise independence is thus far smaller, although still existent.}

We have tested several regularized approximations on the set of examples of Section~\ref{sec:results}, 
including several DCT, wavelet and curvelet approximations (for which we had hoped to 
leverage the recent advances  in compressive sensing and sparse $\ell_1$ methods~\cite{elad}
as well as straightforward piecewise linear data interpolation.
However, the latter is well-known not to be robust against noise, while the former methods are
not suitable in the present context, as they are not built to best take advantage of the
known solution properties. 
The methods which proved winners in the experimentation ultimately use a Tikhonov-type regularization 
in the context of our approximation,  penalizing the discretized $L_2$
integral norm of the gradient or Laplacian of the fields restricted to the boundary segment surface. 
They are further described and theoretically justified in Section~\ref{sec:datainterp}, providing a rare instance
where theory correctly predicts and closely justifies the best practical methods.
We believe that this approach applies to a more general class of PDE-based inverse problems.

In Section~\ref{sec:inverse} we describe the inverse problem and the algorithm
variants used for its solution.
Several aspects arise with the prospect of data completion: which data -- the original or the completed -- to use for
carrying out the iteration, which data for controlling the iterative process, what stopping criterion to use, and more.
These aspects are addressed in Section~\ref{sec:comp_alg}.
The resulting algorithm, based on Algorithm~2 of \cite{rodoas}, is given in Section~\ref{sec:alg}.
The specific EIT/DC resistivity inverse problem described in Section~\ref{sec:detail}
then leads to the data completion methods developed and proved in Section~\ref{sec:datainterp}.

In Section~\ref{sec:results} we apply the algorithm variants developed in the two previous sections
to solve test problems with different receiver locations. 
The purpose is to investigate whether the SS algorithms based on completed data 
achieve results
of similar quality at a cheaper price, as compared to the RS method applied to the original data.
Overall, very encouraging results are obtained even when the original data receiver sets are rather sparse. 
Conclusions are offered in Section~\ref{sec:conclusions}. 
 

\section{Stochastic algorithms for solving the inverse problem}
\label{sec:inverse}

The first two subsections below apply more generally than the third subsection.
The latter settles on one application and leads naturally to
Section~\ref{sec:datainterp}.

Let us recall the acronyms for random subset (RS) and simultaneous sources (SS),
used repeatedly in this section. 


\subsection{Algorithm variants}
\label{sec:comp_alg}

To compare the performance of our model recovery methods with completed data, $\tilde D$, 
against corresponding ones with the original data, $D$,
we use the framework of Algorithm~2 in~\cite{rodoas}. 
This algorithm consists of two stages within each GN iteration. The first stage produces
a stabilized GN iterate, for which we use data denoted by $\hat D$. The second involves assessment
of this iterate in terms of improvement and algorithm termination, using data $\bar D$.
This second stage consists of evaluations of \eqref{phiMW}, in addition to \eqref{1.7}.
We consider three variants:
\begin{enumerate}[(i)]
\item $\hat D = D, \,\, \bar D = D$;\label{alg1_data_i}
\item $\hat D = \tilde D, \,\, \bar D = \tilde D$;\label{alg1_data_ii}
\item $\hat D = \tilde D, \,\, \bar D = D$;\label{alg1_data_iii}
\end{enumerate}
Note that only the RS method can be used in variant~\eqref{alg1_data_i}, whereas any of the SS methods as well
as the RS method can be employed in variant~\eqref{alg1_data_ii}.
In variant~\eqref{alg1_data_iii} we can use a more accurate SS method for the stabilized GN stage and an RS method for
the convergence checking stage, with the potential advantage that the evaluations of
\eqref{phiMW} do not use our ``invented data''. However, the disadvantage is that RS is potentially less
suitable than Gaussian or Hutchinson precisely for tasks such as those in this second stage;
see~\cite{roas}.

A major source of computational expense
is the algorithm stopping criterion,
which in \cite{rodoas} was taken to be
\begin{eqnarray}
\phi(\mm) \leq \rho ,
\label{stop_crit}
\end{eqnarray}
for a specified tolerance $\rho$.
In~\cite{rodoas}, we deliberately employed this criterion
in order to be able to make fair comparisons among different methods.
However, the evaluation of $\phi$ for this purpose is very expensive when $s$ is large,
and in practice $\rho$ is hardly ever known in a rigid sense. 
In any case, this evaluation should be carried out as rarely as possible. 
In~\cite{rodoas}, we addressed this by proposing a safety check, called ``uncertainty check'', 
which uses~\eqref{phiMW} as an unbiased estimator of $\phi(\mm)$ with a stochastic weight matrix 
$W=W^e$ which has far fewer columns than $s$, provided the columns of $W^e$ are independent and 
satisfy~\eqref{w_cond}. 
Thus, in the course of an iteration
we can perform the relatively inexpensive uncertainty check whether    
\begin{equation}
\hat \phi(\mm,W^e) \le \rho .
\label{uncert_check}
\end{equation}
This is like the stopping criterion, but in expectation (with respect to $W^e$). 
If~\eqref{uncert_check} is satisfied, it is an indication that~\eqref{stop_crit} is likely to be satisfied as well, 
so we check the expensive~\eqref{stop_crit} only then. 

In the present article, we propose an alternative heuristic method of 
replacing~\eqref{stop_crit} with another uncertainty check evaluation as 
in~\eqref{uncert_check} with an independently drawn weight matrix $W^e \in \mathbb{R}^{s \times t_{n}}$, 
whose $t_{n}$ columns have i.i.d. elements drawn from the Rademacher distribution 
(NB the Hutchinson estimator has smaller variance than Gaussian). 
The sample size $t_{n}$ can be heuristically set as 
\begin{equation}
t_{n} = \min{(s,\max{(t_{0},s_{n})})} ,
\label{soft_stop_crit_t0}
\end{equation}
where $t_{0} > 1$ is some preset minimal sample size for this purpose.
Thus, for each algorithm variant~\eqref{alg1_data_i},~\eqref{alg1_data_ii} or~\eqref{alg1_data_iii}, 
we consider two stopping criteria, namely,
\begin{enumerate}[(a)]
\item the hard \eqref{stop_crit}, and \label{alg1_stop_a} 
\item the more relaxed \eqref{uncert_check}+\eqref{soft_stop_crit_t0}. \label{alg1_stop_b}
\end{enumerate}

When using the original data $D$ in the second stage of our general algorithm, 
as in variants~\eqref{alg1_data_i} and~\eqref{alg1_data_iii} above, since the projection matrices $P_{i}$ 
are not the same across experiments, 
one is restricted to the RS method as an unbiased estimator. 
However, when the completed data is used and we only have one $P$ for all experiments, 
we can freely use the stochastic SS methods and leverage their rather better accuracy
in order to estimate the true misfit $\phi(\mm)$. 
This is indeed an important advantage of data completion methods.
 
However, when using the completed data $\tilde D$ in the second stage of our general algorithm, 
as in variant~\eqref{alg1_data_ii}, an issue arises: when the data is completed, the given tolerance $\rho$ 
loses its meaning and we 
need to take into account the effect of the additional data to calculate a new tolerance. 
Our proposed heuristic approach is to replace $\rho$ with a new tolerance
${\rho} := (1+c)\rho$, where $c$ is the percentage of the data that needs to be completed
expressed as a fraction. 
For example, if $30\%$ of data is to be completed then we set ${\rho} := 1.3 \rho$. 
Since the completed data after using~\eqref{eqn_sub_01_datainterp} or~\eqref{eqn_sub_02_datainterp} 
is smoothed and denoised, we only need to add a small fraction of the initial tolerance to get the new one, 
and in our experience, $1+c$ is deemed to be a satisfactory factor. 
We experiment with this less rigid stopping criterion in Section~\ref{sec:results}.


\subsection{General algorithm}
\label{sec:alg}

Our general algorithm utilizes a stabilized Gauss-Newton (GN) method~\cite{doas3}, where each
iteration consists of two stages as described in Section~\ref{sec:comp_alg}.
In addition to combining the elements described above, this algorithm also provides
a schedule for selecting the sample size $s_n$ in the $n$th stabilized GN iteration.
In Algorithm~\ref{alg1}, variants~\eqref{alg1_data_i},~\eqref{alg1_data_ii} and~\eqref{alg1_data_iii}, and criteria~\eqref{alg1_stop_a} and~\eqref{alg1_stop_b},
are as specified in Section~\ref{sec:comp_alg}.

\begin{algorithm}[htb]
\caption{Solve inverse problem using variant~\eqref{alg1_data_i},~\eqref{alg1_data_ii} or~\eqref{alg1_data_iii}, cross validation, 
and stopping criterion~\eqref{alg1_stop_a} or~\eqref{alg1_stop_b}}
\begin{algorithmic}
\STATE \textbf{Given:} sources $Q$, 
measurements  $\hat D$, measurements $\bar D$,
stopping tolerance $\rho$, decrease factor $\kappa < 1$, and initial guess $\mm_{0}$.
\STATE \textbf{Initialize:} $\mm = \mm_{0} \; , \; s_{0} = 1$.
\FOR {$n = 0,1,2, \ldots$ until termination} 
\STATE - Choose a $W^{f}_{n} \in \R^{s \times s_n}$ stochastically from appropriate distribution.
\STATE - \textbf{Fitting}: Perform one stabilized GN iteration, based on $\hat D$, with $W=W^{f}_{n}$.
\STATE - Choose $W^{c}_{n} \in \R^{s \times s_n}$ and $W^{e}_{n} \in \R^{s \times s_n}$ stochastically from appropriate distribution.
\IF {$\hat \phi(\mm_{n+1},W^{c}_{n}) \leq \kappa \hat \phi(\mm_{n},W^{c}_{n})$, based on $\bar D$,
i.e., \textbf{Cross Validation} holds}
\STATE - \textbf{Uncertainty Check}: Compute~\eqref{phiMW} based on $\bar D$ using $\mm_{n+1}$ and $W^{e}_{n}$.
\IF {\eqref{uncert_check} holds}
\STATE - \textbf{Stopping Criterion}: 
\IF {Option~\eqref{alg1_stop_a} selected and \eqref{stop_crit} holds} 
\STATE {\bf terminate}; otherwise set $s_{n+1} = s_n$.
\ELSE 
\STATE Set $t_{n} = \min{(s,\max{(t_{0},s_{n})})}$.
\STATE Draw another $W^{e}_{n} \in \R^{s \times t_n}$ stochastically from appropriate distribution.
{\bf Terminate} if \eqref{uncert_check} holds using $\bar D$; otherwise set $s_{n+1} = s_n$.
\ENDIF
\ENDIF
\ELSE
\STATE - \textbf{Sample Size Increase}: for example, set $s_{n+1} = \min(2 s_{n},s)$.
\ENDIF
\ENDFOR
\end{algorithmic}
\label{alg1}
\end{algorithm}

\add{For implementation details as well as convergence and regularization results pertaining to this algorithm,
we refer to~\cite{rodoas,doas3} and references therein.}


\subsection{The DC resistivity inverse problem}
\label{sec:detail}

For the forward problem we consider, following \cite{HaberChungHermann2010,doas3,rodoas}, a linear PDE of the form
\begin{subequations}
\begin{eqnarray}
\div (\sigma(\xx) \grad u) = q(\xx), \quad \xx \in  \Omega ,
\label{2.1a}
\end{eqnarray}
where $\sigma$ is a given conductivity function which may be rough (e.g., discontinuous) but is bounded away from $0$: there is
a constant $\sigma_0 > 0$ such that $\sigma (\xx) \geq \sigma_0 , \; \forall \xx \in \Omega$.
\add{A similar PDE is used for the EIT problem.}
This elliptic PDE is subject to the homogeneous Neumann boundary conditions
\begin{eqnarray}
\frac{\partial u}{\partial n} = 0, \quad \xx \in \partial\Omega. \label{2.1b}
\end{eqnarray}
\label{2.1}
\end{subequations}

In our numerical examples we take $\Omega \subset \R^d$ to be the unit square or unit cube, 
and the sources $\qq$ to be the differences of $\delta$-functions.
Furthermore, the receivers (where data values are measured) lie in $\partial \Omega$,
so in our data completion algorithms we approximate data along one of four edges in the 2D
case or within one of six square faces in the 3D case.
\add{The setting of our experiments, which follows that used in~\cite{rodoas}, 
is more typical of DC resistivity than of the EIT problem.}   

For the inverse problem we introduce additional a priori information, when such is available, via a 
point-wise parameterization of $\sigma(\xx)$ in terms of $m(\xx)$.
Define the transfer function
\begin{eqnarray}
\psi(\tau) = \psi (\tau; \theta,\alpha_{1},\alpha_{2}) = 
\alpha \tanh\left(\frac{\tau}{\alpha\theta}\right) + \frac{\alpha_{1} +\alpha_{2}}2 , 
\quad \alpha = \frac{\alpha_{2} -\alpha_{1}}2 \label{2.5}.
\end{eqnarray}
If we know  that the sought conductivity function $\sigma (\xx)$ 
takes only one of two values, $\sigma_I$ or $\sigma_{II}$, at each $\xx$, 
then we use an approximate level set function representation, 
writing $\sigma (\xx) = \lim_{h \rightarrow 0} \sigma(\xx;h)$, where
\begin{eqnarray}
\sigma(\xx;h) = \psi(m(\xx ); h, \sigma_{I}, \sigma_{II} ) . \label{2.7}
\end{eqnarray}
The function $\psi$ here depends on the resolution, or grid width $h$.
More commonly, we may only know reasonably tight bounds, say $\sigma_{\min}$ and $\sigma_{\max}$,
such that $ \sigma_{\min} \leq \sigma (\xx) \leq \sigma_{\max}$.
Such information may be enforced using \eqref{2.5} by defining
\begin{eqnarray}
 \sigma(\xx) = \psi (m(\xx)), \quad {\rm with~} \psi(\tau) = \psi(\tau; 1, \sigma_{\min}, \sigma_{\max} ) . \label{2.6}
 \end{eqnarray}
For details of this, as well as the PDE discretization and the {\em stabilized} GN iteration
used, we refer to~\cite{doas3,rodoas} and references therein.


\section{Data completion}
\label{sec:datainterp}
Let $\Lambda_{i}\subset \partial\Omega$ denote the point set of receiver
locations for the $i^{th}$ experiment. Our goal here is to extend 
the data for each experiment to the union
$\Lambda = \bigcup_i \Lambda_{i} \subseteq \partial\Omega$, the common measurement domain.
To achieve this, we choose a suitable boundary patch $\Gamma \subseteq \partial\Omega$, 
such that $\Lambda \subset \bar{\Gamma}$, where $\bar{\Gamma}$ denotes the closure of 
$\Gamma$ with respect to the bounbdary subspace topology. 
For example, one can choose $\Gamma$ to be the interior of the 
convex hull (on $\partial \Omega$) of $\Lambda$. 
We also assume that $\Gamma$ can be selected such that it is a simply connected open set.
For each experiment $i$, we then construct an extension function $v_i$ on $\bar{\Gamma}$
which approximates the measured data on $\Lambda_i$. The
extension method can be viewed as an inverse problem, and we select
a regularization based on knowledge of the function space that $v_i$
(which represents the restriction of potential $u_i$ to $\Gamma$) should
live in.
Once $v_i$ is constructed, the extended data, $\tilde{\dd}_i$, is
obtained by restricting $v_i$ to $\Lambda$, denoted in what follows by
$v_{i}^\Lambda$. Specifically, for the receiver location $x_{j} \in
\Lambda$, we set $[\tilde{\dd}_{i}]_{j} = v_{i}(x_{j})$, where
$[\tilde{\dd}_{i}]_{j}$ denotes the $j^{th}$ component of vector
$\tilde{\dd}_{i}$ corresponding to $x_{j}$. 
Below we show that the trace of potential $u_i$ to the boundary is indeed continuous, 
thus point values of the extension function $v_i$ make sense.

In practice, the conductivity $\sigma(\xx)$ in~\eqref{2.1a} is often piecewise smooth with finite jump discontinuities. 
As such one is faced with two scenarios leading to two approximation
methods for finding $v_i$: 
(a) the discontinuities are some distance away from $\Gamma$; and 
(b) the discontinuities extend all the way to $\Gamma$. 
These cases result in a different a priori smoothness of the field $v_i$ on $\Gamma$. 
Hence, in this section we treat these cases separately and propose an appropriate data completion algorithm for each.

Consider the problem~\eqref{2.1}.
In what follows we assume that $\Omega$ is a bounded open domain and $\partial\Omega$ is Lipschitz.
Furthermore, we assume that $\sigma$ is continuous on  a finite number 
of disjoint subdomains, $\Omega_{j} \subset \Omega$, such that 
$\bigcup_{j=1}^{N} \overline{\Omega}_{j} = \overline{\Omega}$ and
$\partial{\Omega}_{j} \cap \overline{\Omega} \in C^{2,\alpha}$,
for some $0 < \alpha \leq 1$,
i.e., $\sigma \in C^{2}(\overline{\Omega}_{j}), \;  j = 1,\ldots ,N$.\footnote{$\overline{X}$ denotes the closure of $X$ with respect to the appropriate topology.}
Moreover, assume that $q \in L_{\infty}(\Omega)$ and $q \in \text{Lip}(\overline{\Omega}_{j} \cap \Omega)$, i.e., 
it is Lipschitz continuous in each subdomain; 
this assumption will be slightly weakened in Subsection~\ref{point_source}.

Under these assumptions and for the Dirichlet problem with a $C^{2}(\partial \Omega)$ boundary condition, 
there is a constant $\gamma$, $0 < \gamma \leq 1$, 
such that $u \in C^{2,\gamma}(\overline{\Omega}_{j})$~\cite[Theorem 4.1]{isakov}. 
In~\cite[Corollary 7.3]{LiVog}, it is also shown that 
the solution on the entire domain is H\"{o}lder continuous, 
i.e., $u \in C^{\beta}(\overline{\Omega})$ for some $\beta$, $0 < \beta \leq 1$. 
Note that the mentioned theorems are stated for the Dirichlet problem, and in the present article 
we assume a homogeneous Neumann boundary condition. 
However, in this case we have infinite smoothness in the normal direction at the boundary, 
i.e., $C^{\infty}$ Neumann condition, 
and no additional complications arise; 
see for example~\cite{Shkoller}. So the results stated above would still hold for \eqref{2.1}.

 
\subsection{Discontinuities in conductivity are away from common measurement domain}
\label{subsec_01_datainterp}

This scenario corresponds to the case where the boundary patch $\Gamma$ can be chosen 
such that $\Gamma  \subset ( \partial{\Omega}_{j} \cap \partial\Omega )$  for some $j$. 
Then we can expect a rather smooth field at $\Gamma$; 
precisely, $u \in C^{2,\gamma}(\overline{\Gamma})$. 
Thus, $u$ belongs to the Sobolev space $H^{2}(\Gamma)$, and we can impose this knowledge in our 
continuous completion formulation. 
For the $i^{th}$ experiment, we define our data completion function 
$v_i \in H^{2}(\Gamma) \cap C(\overline{\Gamma})$ as
\begin{equation}
v_i =  \argmin{v}\ \ \frac{1}{2} \| v^{\Lambda_{i}} - \dd_{i} \|^2_2 + \lambda \left\|\Delta_S v\right\|_{L_{2}(\Gamma)}^{2},
\label{eqn_sub_01_datainterp}
\end{equation}
where $\Delta_S$ is the Laplace-Beltrami operator for the Laplacian on
the boundary surface and $v^{\Lambda_{i}}$ is the restriction of the continuous function $v$ to the point set $\Lambda_{i}$.
The regularization parameter $\lambda$ depends on the amount of noise in our data;
see Section~\ref{choose_lambda}.

We next discretize~\eqref{eqn_sub_01_datainterp} using
a mesh on $\Gamma$ as specified in Section~\ref{sec:results},
and solve the resulting linear least squares problem using standard techniques. 

Figure~\ref{fig_lap_interp} shows an example of such data completion. 
The true field and the measured data correspond to an experiment described in Example~\ref{exp03} 
of Section~\ref{sec:results}. 
We only plot the profile of the field along the top boundary of the 2D domain. 
As can be observed, the approximation process
imposes smoothness which results in an excellent completion of the missing data, despite the presence of 
noise at a fairly high level. 

\begin{figure}[htb]\centering
 {\includegraphics[width=.6\linewidth]{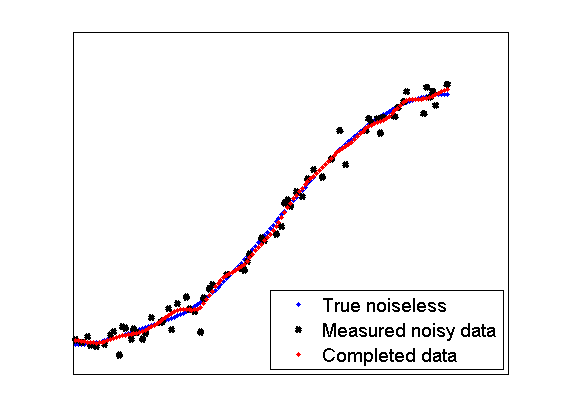}}
    \caption{Completion using the regularization~\eqref{eqn_sub_01_datainterp}, for 
an experiment taken from Example~\ref{exp03} where $50\%$ of 
the data requires completion and the noise level is $5\%$. 
Observe that even in the presence of significant noise, the data completion formulation~\eqref{eqn_sub_01_datainterp} 
achieves a good quality field reconstruction.
\label{fig_lap_interp}}
\end{figure}

We hasten to point out that the results in Figure~\ref{fig_lap_interp}, 
as well as those in Figure~\ref{fig_grad_interp} below, 
pertain to differences in field values, i.e., the solutions of the forward problem $u_i$, 
and not those in the inverse problem solution shown, e.g., in Figure~\ref{fig05}.
The good quality approximations in Figures~\ref{fig_lap_interp} and~\ref{fig_grad_interp}
generally form a necessary but not sufficient condition for success in the inverse problem
solution.


\subsection{Discontinuities in conductivity extend all the way to common measurement domain}
\label{subsec_02_datainterp}

This situation corresponds to the case in which $\Gamma$ can only be chosen such that it intersects more 
than just one of the 
$(\partial{\Omega} \cap \partial{\Omega}_{j})$'s. More precisely, assume that there is an index set 
$\mathcal{J} \subseteq \{1,2, \cdots N\}$ with $|\mathcal{J}| = K \geq 2$  such that 
$\left\{\Gamma \cap (\partial{\Omega} \cap \partial{\Omega}_{j})^{\circ} \;\;,\;\;  j \in \mathcal{J}\right\}$ 
forms a set of disjoint subsets of $\Gamma$ such that 
$\overline{\Gamma} = \bigcup_{j \in \mathcal{J}} \overline{\Gamma \cap ( \partial{\Omega} \cap \partial{\Omega}_{j} )^{\circ}}$,  
where $X^{\circ}$ denotes the interior of the set $X$, 
and that the interior is with respect to the subspace topology on $\partial \Omega$. 
In such a case $u$, restricted to $\Gamma$, is no longer necessarily in $H^{2}(\Gamma)$. Hence, 
the smoothing term in~\eqref{eqn_sub_01_datainterp} is no longer valid, 
as $\left\|\Delta_S u \right\|_{L_{2}(\Gamma)}$ might be undefined or infinite. 
However, as described above, we know that the solution is piecewise smooth and overall continuous, 
i.e., $u \in C^{2,\gamma}(\overline{\Omega}_{j})$ and $u \in C^{\beta}(\overline{\Omega})$. 
The following theorem shows that the smoothness on $\Gamma$ is not completely gone:
we may lose one degree of regularity at worst. 

\begin{theorem}
\label{thm01}
Let $U$ and $\{U_{j} | \; j = 1,2,\ldots ,K\}$ be open and bounded sets such that the $U_j$ are 
pairwise disjoint and $\overline{U} = \bigcup_{j=1}^{K} \overline{U}_{j}$. 
Further, let $u \in C(\overline{U}) \cap H^{1}(U_{j}) \; \forall j$. 
Then $u \in H^{1}(U)$.
\end{theorem}

\begin{proof}
It is easily seen that since $u \in C(\overline{U})$ and $U$ is bounded, then $u \in L_{2}(U)$. 
Now, let $\phi \in C^{\infty}_{0} (U)$ be a test function and denote $\partial_{i} \equiv \frac{\partial}{\partial \xx_{i}}$. 
Using the assumptions that the $U_{j}$'s form a partition of $U$, 
$u$ is continuous in $\overline{U}$, 
$\phi$ is compactly supported inside $U$, and the fact that the $\partial U_{j}$'s have measure zero, we obtain
\begin{eqnarray*}
\label{3.1}
\int_{U}{u \partial_{i} \phi} = \int_{\overline{U}}{u \partial_{i} \phi} = \int_{\cup_{j = 1}^{K} \overline{U_{j}}}{u \partial_{i} \phi} 
= \int_{(\cup_{j = 1}^{K} U_{j}) \bigcup (\cup_{j = 1}^{K} \partial U_{j})}{u \partial_{i} \phi} \\
= \int_{\cup_{j = 1}^{K} U_{j}}{u \partial_{i} \phi} = \sum_{j = 1}^{K} \int_{U_{j}}{u \partial_{i} \phi} 
= \sum_{j = 1}^{K} \int_{\partial U_{j}}{u \phi \nu_{i}^{j}} - \sum_{j = 1}^{K} \int_{U_{j}}{\partial_{i} u \phi}, \\
\end{eqnarray*}
where $\nu_{i}^{j}$ is the $i^{th}$ component of the outward unit surface normal to $\partial U_{j}$. 
Since $u \in H^{1}(U_{j}) \; \forall j$, the second part of the rightmost expression makes sense. 
Now, for two surfaces $\partial U_{m}$ and $\partial U_{n}$ such that
$\partial U_{m} \cap \partial U_{n} \neq \emptyset$, 
we have $\nu_{i}^{m}(\xx) = -\nu_{i}^{n}(\xx) \; \forall \xx \in \partial U_{m} \cap \partial U_{n}$. 
This fact, and noting in addition that $\phi$ is compactly supported inside $U$, 
makes the first term in the right hand side vanish. 
We can now define the weak derivative of $u$ with respect to $\xx_{i}$ to be
\begin{equation}
\label{3.2}
v(\xx) = \sum_{j = 1}^{K} \partial_{i} u \mathcal{X}_{U_{j}} ,
\end{equation}
where $\mathcal{X}_{U_{j}}$ denotes the characteristic function of the set $U_{j}$. 
This yields
\begin{equation}
\label{3.3}
\int_{U}{u \partial_{i} \phi} = - \int_{U}{v \phi}.
\end{equation}
Also
\begin{equation}
\label{3.4}
\| v \|_{L_{2}(U)} \leq \sum_{j = 1}^{K} \| \partial_{i} u \|_{L_{2}(U_{j})}  < \infty, \\
\end{equation}
and thus we conclude that $u \in H^{1}(U)$.
$\qed$
\end{proof}

If the assumptions stated at the beginning of this section hold then  
we can expect a field $u \in H^{1}(\Gamma) \cap C(\bar{\Gamma})$. This is obtained by invoking Theorem~\ref{thm01} 
with $U = \Gamma$ and $U_{j} = \Gamma \cap (\partial{\Omega} \cap \partial{\Omega}_{j})^{\circ}$ for all $j \in \mathcal{J}$.

Now we can formulate the data completion method as
\begin{equation} 
v_i =  \argmin{v}\ \ \frac{1}{2} \| v^{\Lambda_{i}} - \dd_{i} \|^2_2 + 
\lambda \left\|\grad_S v\right\|_{L_{2}(\Gamma)}^{2} \label{eqn_sub_02_datainterp},
\end{equation}
where $v^{\Lambda_{i}}$ and $\lambda$ are as in Section~\ref{subsec_01_datainterp}.

Figure~\ref{fig_grad_interp} shows an example of data completion using the formulation~\eqref{eqn_sub_02_datainterp},
depicting the profile of $v_i$ along the top boundary.
The field in this example is continuous and only piecewise smooth.  
The approximation process imposes less smoothness along the boundary as compared to~\eqref{eqn_sub_01_datainterp},   
and this results in an excellent completion of the missing data, despite a nontrivial level of noise. 

\begin{figure}[htb]\centering		
{\includegraphics[width=.6\linewidth]{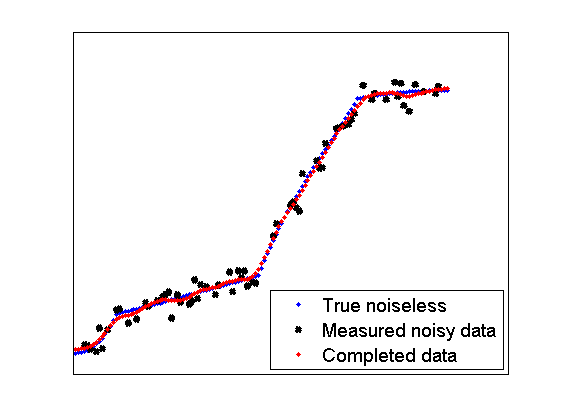}}
    \caption{Completion using the regularization~\eqref{eqn_sub_02_datainterp}, for
an experiment taken from Example~\ref{exp02} where $50\%$ of the data 
		requires completion and the noise level is $5\%$. 
		Discontinuities in the conductivity extend to the measurement domain and their effect on the field profile
		along the boundary can be clearly observed. Despite the large amount of noise, 
		data completion formulation~\eqref{eqn_sub_02_datainterp} achieves a good reconstruction.
\label{fig_grad_interp}}
\end{figure}
 
To carry out our data completion strategy, the problems
\eqref{eqn_sub_01_datainterp} or~\eqref{eqn_sub_02_datainterp} are discretized.
This is followed by a straightforward linear least squares technique, 
which can be carried out very efficiently.
Moreover, this is a preprocessing stage performed once, which is completed before the 
algorithm for solving the nonlinear
inverse problem commences. 
Also, as the data completion for each experiment can be carried out independently of others, 
the preprocessing stage can be done in parallel if needed.
Furthermore, the length of the vector of unknowns $v_i$ is relatively small compared to those of $u_i$ 
because only the boundary is involved. All in all the amount of work involved in the 
data completion step is dramatically less than one full evaluation of
the misfit function~\eqref{1.1}.

 
\subsection{Determining the regularization parameter} 
\label{choose_lambda}

Let us write the discretization of~\eqref{eqn_sub_01_datainterp} or~\eqref{eqn_sub_02_datainterp} as
\begin{equation}
\min_{\vv}  \hf \| \hat{P}_{i} \vv - \dd_{i} \|_{2}^{2} + \lambda \|L \vv \|_{2}^2 ,
\label{data_completion_discretized}
\end{equation}
where $L$ is the discretization of the surface gradient or Laplacian operator, 
$\vv$ is a vector whose length is the size of the discretized 
$\Gamma$, $\hat{P}_{i}$ is the projection matrix
from the discretization of $\Gamma$ to $\Lambda_{i}$, and $\dd_{i}$ is the $i^{th}$ original measurement vector.
 
Determining $\lambda$ in this context is a textbook problem; see, e.g.,~\cite{vogelbook}.
Viewing it as a parameter, we have a linear least squares problem for $\vv$ in \eqref{data_completion_discretized},
whose solution can be denoted $\vv (\lambda)$.
Now, in the simplest case, which we assume in our experiments,
the noise level for the $i^{th}$ experiment, $\eta_{i}$, is known, so one can use 
the discrepancy principle to pick $\lambda$ such that
\begin{equation}
\left\| \hat{P}_{i} \vv(\lambda) - \dd_{i} \right\|_{2}^{2} \leq \eta_{i} .
\label{disc_princ}
\end{equation}
Numerically, this is done by setting equality in \eqref{disc_princ} and solving the resulting nonlinear equation
for $\lambda$ using a standard root finding technique.

If the noise level is not known, one can use the generalized cross validation (GCV) method or the L-curve method; see~\cite{vogelbook}.
We need not dwell on this longer here.

 
\subsection{Point sources and boundaries with corners}
\label{point_source}
In the numerical examples of Section~\ref{sec:results}, as in \cite{doas3,rodoas},
we use delta function combinations as the sources $q_i(\xx)$, 
\add{in a manner that is typical in exploration geophysics 
(namely, DC resistivity as well as low-frequency electromagnetic experiments), 
less so in EIT. However,} 
these are clearly not honest $L_{\infty}$ functions.
Moreover, our domains $\Omega$ are a square or a cube and as such they have corners.
 
However, the theory developed above, and the data completion methods that it generates, can be extended to our 
experimental setting because we have control over the experimental setup.
The desired effect is obtained by simply separating the location of each source from any of the receivers,
and avoiding domain corners altogether.

Thus, consider in \eqref{2.1a} a source function of the form
\begin{eqnarray*} 
q(\xx) = \hat q(\xx) + \delta (\xx-\xx^*) - \delta (\xx - \xx^{**}),
\end{eqnarray*}
where $\hat q$ satisfies the assumptions previously made on $q$. Then we select
$\xx^*$ and $\xx^{**}$ such that there are two open balls $B(\xx^*,r)$ and $B(\xx^{**},r)$
of radius $r > 0$  each and centered at $\xx^*$ and $\xx^{**}$, respectively, where
(i) no domain corner belongs to $B(\xx^*,r)\cup B(\xx^{**},r)$, and (ii) $(B(\xx^*,r)\cup B(\xx^{**},r)) \cap \Gamma$ is empty.
Now, in our elliptic PDE problem the lower smoothness effect of either a domain corner
or a delta function is local! 
In particular, the contribution of the point source to the flux $\sigma \grad u$
is the integral of $\delta (\xx-\xx^*) - \delta (\xx - \xx^{**})$, and this is smooth outside the union of the two balls.



\section{Numerical experiments}
\label{sec:results}

The PDE problem used in our experiments is described in Section~\ref{sec:detail}.
For each experiment $i$ there is a positive unit point source at $\xx^{i}_1$ and a negative sink 
at $\xx^{i}_{2}$, where $\xx^{i}_{1}$ and $\xx^{i}_{2}$ are two locations on the 
boundary $\partial \Omega$. 
Hence in \eqref{2.1a} we must consider sources of the form 
$q_i(\xx) = \delta(\xx-\xx_{1}^i) - \delta(\xx-\xx_{2}^i)$, 
i.e., a difference of two $\delta$-functions.
For our experiments in 2D, when we place a source on the left boundary, 
the corresponding sink on the right boundary is placed in every possible combination. 
Hence, having $p$ locations on the left boundary for the source would result in $s = p^2$ experiments.
The receivers are located at the top and bottom boundaries. As such, the completion steps~\eqref{eqn_sub_01_datainterp} or~\eqref{eqn_sub_02_datainterp} are carried out separately for the top and bottom boundaries. No source or receiver is placed at the corners.
In 3D we use an arrangement whereby 
four boreholes are located at the four edges of the cube, and source and sink pairs are put
at opposing boreholes in every combination, except that there are no
sources on the point of intersection of boreholes and the surface, 
i.e., at the top four corners, since these four nodes are part of the surface where data values are gathered.

In the sequel we generate data $\dd_i$ by using a chosen true model (or ground truth)
and a source-receiver configuration as described above. 
Since the field $u$ from~\eqref{2.1} is only determined up to a constant, only voltage
differences are meaningful. Hence we subtract for each $i$
the average of the boundary potential values
from all field values
at the locations where data is measured. As a result each row of the
projection matrix $P_i$ has zero sum.
This is followed by peppering these values with additive Gaussian noise to create the data $\dd_i$
used in our experiments.
Specifically, for an additive noise of $2\%$, say, denoting the ``clean data'' $l \times s$ matrix by $D^*$,
we reshape this matrix into a vector $\dd^*$ of length $sl$, calculate the standard deviation 
${\tt sd} = .02\| \dd^* \|/\sqrt{sl}$, and define $D = D^* + {\tt sd * randn(l,s)}$
using {\sc Matlab}'s random generator function {\tt randn}.


For all of our numerical experiments, the ``true field'' is calculated on a grid that 
is twice as fine as the  one used to reconstruct the model. 
For the 2D examples, the reconstruction is done on a uniform grid of size $129^2$
with $s=961$ experiments in the setup described above. 
For the 3D examples, we set $s = 512$ and employ a uniform grid of size $33^{3}$, except for  Example~\ref{exp03}
where the grid size is $17^3$.

\add{In the numerical examples considered below, we use true models with piecewise constant levels,
with the conductivities bounded away from $0$. 
For further discussion of such models within the context of EIT, see~\cite{gkljskm12}.}

Numerical examples are presented for both cases described in Sections~\ref{subsec_01_datainterp} and~\ref{subsec_02_datainterp}. 
For all of our numerical examples except Examples~\ref{exp05} and~\ref{exp06}, we use the transfer function~\eqref{2.6} 
with $\sigma_{\max} = 1.2 \max \sigma(\xx)$, 
and $\sigma_{\min} = \frac{1}{1.2} \min \sigma(\xx)$. In the ensuing calculations we then ``forget'' 
what the exact $\sigma (\xx)$ is.
Further, in the stabilized GN iteration we employ preconditioned conjugate gradient (PCG)
inner iterations, setting \add{as in \cite{rodoas}} the PCG iteration limit to $r=20$, and the PCG tolerance to $10^{-3}$. 
The initial guess is $\mm_{0} = {\bf 0}$. Examples~\ref{exp05} and~\ref{exp06} are carried out 
using the level set method~\eqref{2.7}. 
Here we can set $r = 5$, significantly lower than above. The initial guess for the level set example is 
a cube with rounded corners inside $\Omega$ (see Figure~2 in~\cite{rodoas}). 

For Examples~\ref{exp01},~\ref{exp02},~\ref{exp03} and~\ref{exp05}, in addition to displaying the log conductivities 
(i.e., $\log (\sigma)$) for each reconstruction, we also show the log-log plot of misfit on the 
entire data (i.e., $\| F(\mm) - D \|_F$) vs. PDE count. 
A table of total PDE counts (not including what extra is required for the plots)
for each method is displayed. 
In order to simulate the situation where sources do not share the same receivers, 
we first generate the data fully on the entire domain of measurement 
and then knock out at random some percentage of the generated data.
This setting roughly corresponds to an EMG experiment with faulty receivers.  


For each example, we use Algorithm~\ref{alg1} with one of the variants~\eqref{alg1_data_i},~\eqref{alg1_data_ii} 
or~\eqref{alg1_data_iii} paired with one of the stopping criteria~\eqref{alg1_stop_a} or~\eqref{alg1_stop_b}. 
For instance, when using variant~\eqref{alg1_data_ii} with the soft stopping criterion~\eqref{alg1_stop_b}, 
we denote the resulting algorithm by (\ref{alg1_data_ii},~\ref{alg1_stop_b}). 
For the relaxed stopping rule~\eqref{alg1_stop_b} we (conservatively) set $t_{0} = 100$ 
in~\eqref{soft_stop_crit_t0}.
A computation using RS applied to the original data, using variant (i,x), is compared to one using SS applied
to the completed data through variant (ii,x) or (iii,x), where x stands for~\ref{alg1_stop_a} or~\ref{alg1_stop_b}.

For convenience of cross reference, we gather all resulting seven algorithm comparisons and corresponding work counts
in Table~\ref{table01} below. For Examples~\ref{exp01},~\ref{exp02},~\ref{exp03} and~\ref{exp05}, 
the corresponding entries of this table should be read together with the misfit plots for each example.  

\begin{table}[!ht]
\begin{center}
\begin{tabular}{|l|l|cc|}
\hline 
{\bf Example} &  Algorithm & Random Subset & Data Completion
\\ \hline 
\ref{exp01} & $(i,a) ~|~ (iii,a)$ & 3,647 & 1,716 \\ \hline
\ref{exp02} & $(i,a) ~|~ (iii,a)$ & 6,279 & 1,754 \\ \hline
\ref{exp03} & $(i,a) ~|~ (iii,a)$ & 3,887 & 1,704 \\ \hline
\ref{exp04} & $(i,b) ~|~ (ii,b)$  & 4,004 & 579 \\ \hline 
\ref{exp05} & $(i,a) ~|~ (iii,a)$ & 3,671 & 935 \\ \hline
\ref{exp06} & $(i,b) ~|~ (ii,b)$  & 1,016 & 390 \\ \hline
\ref{exp07} & $(i,b) ~|~ (ii,b)$  & 4,847 & 1,217 \\ \hline
\end{tabular}
\end{center}
\caption{Algorithm and work in terms of number of PDE solves, comparing RS against data completion using Gaussian SS.
\label{table01}}
\end{table} 

\begin{example}
\label{exp01} 
In this example, we place two target objects of conductivity $\sigma_I = 1$ in a background of 
conductivity $\sigma_{II} = 0.1$, 
and $5\%$ noise is added to the data as described above. Also, $25\%$ of the data requires completion.
The discontinuities in the conductivity are touching the measurement domain, 
so we use~\eqref{eqn_sub_02_datainterp} to complete the data.
The hard stopping criterion~\eqref{alg1_stop_a} is employed, and iteration control is done using the original data,
i.e., variants (\ref{alg1_data_i},~\ref{alg1_stop_a}) and (\ref{alg1_data_iii},~\ref{alg1_stop_a})
are compared: see the first entry of Table~\ref{table01} and Figure~\ref{fig06}(a).

\begin{figure}[htb]\centering
     \begin{minipage}[htb]{0.3\linewidth}\centering
        \subfigure[True model]		{\includegraphics[width=.98\linewidth]{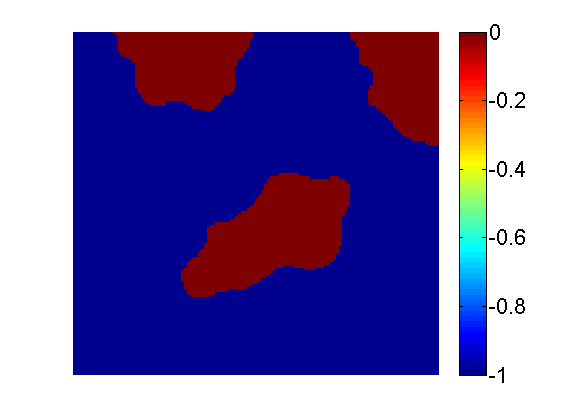}}
    \end{minipage}
    \begin{minipage}[htb]{0.3\linewidth}\centering
        \subfigure[Random Subset] {\includegraphics[width=.98\linewidth]{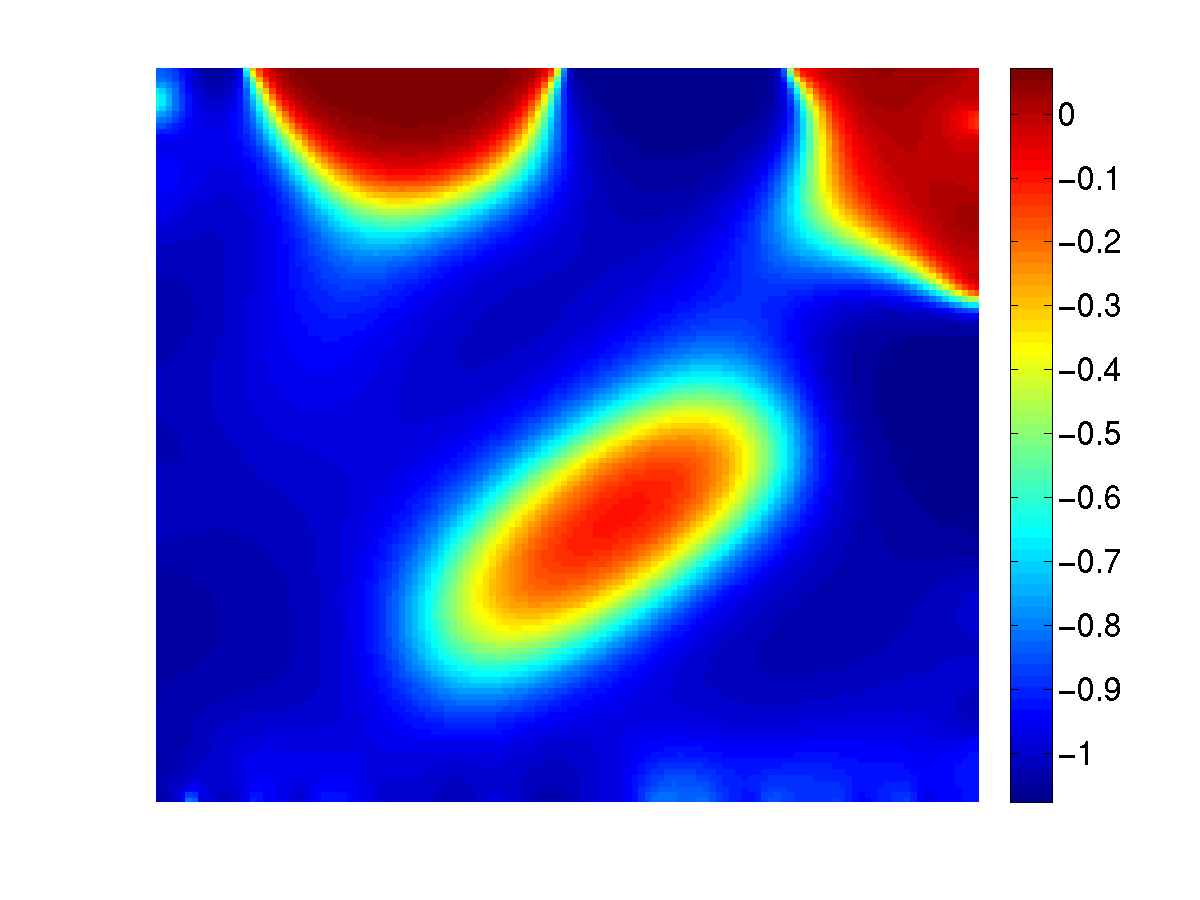}}
    \end{minipage}
    \begin{minipage}[htb]{0.3\linewidth}\centering
        \subfigure[Data Completion] {\includegraphics[width=.98\linewidth]{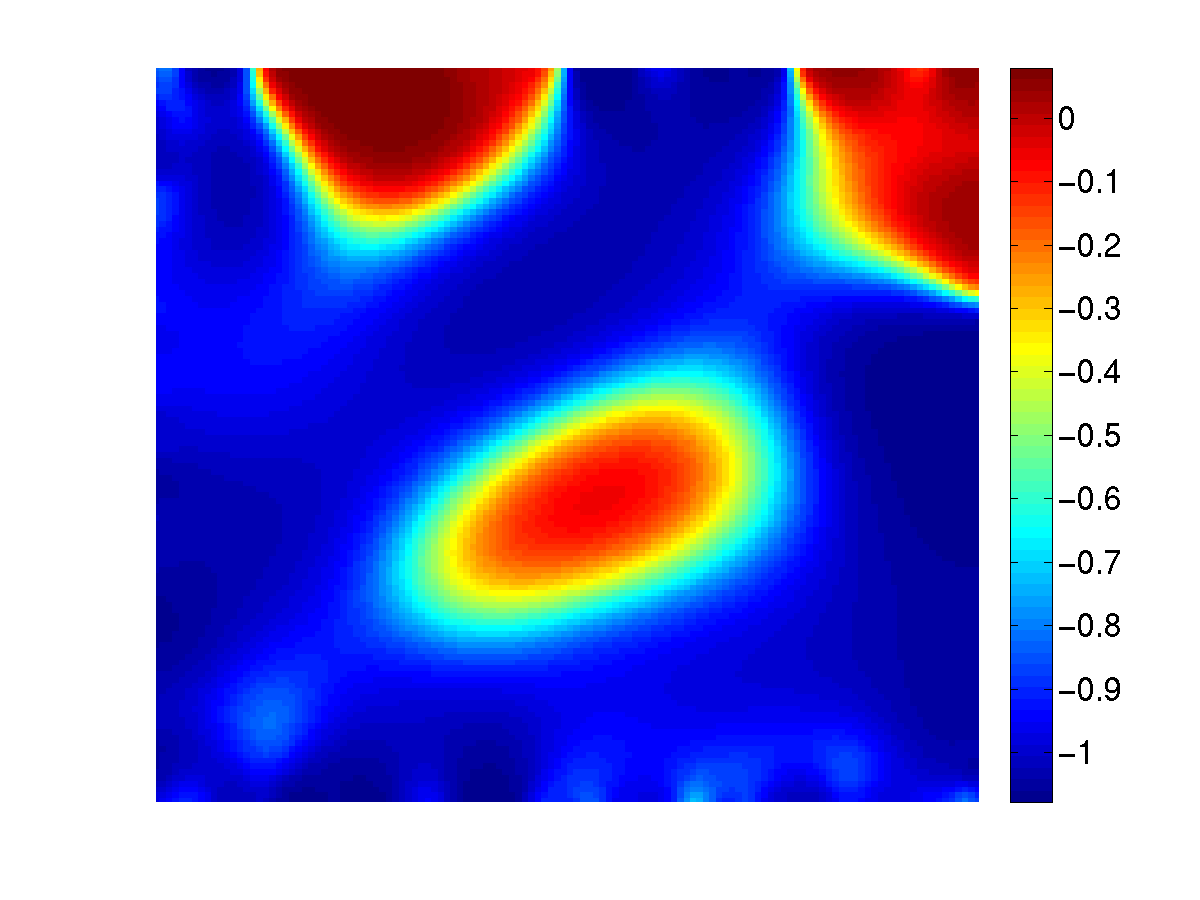}}
    \end{minipage}\hfill
    \caption{Example~\ref{exp01} -- reconstructed log conductivity with $25\%$ data missing and $5\%$ noise. 
	Regularization~\eqref{eqn_sub_02_datainterp} has been used to complete the data. \label{fig03}}
\end{figure}

The corresponding reconstructions are depicted in Figure~\ref{fig03}. It can be seen that roughly the same
quality reconstruction is obtained using the data completion method at less than half the price.
\end{example}

\begin{example}
\label{exp02}
This example is the same as Example~\ref{exp01}, except that $50\%$ of the data is missing and requires completion. 
The same algorithm variants as in Example~\ref{exp01} are compared.
The reconstructions are depicted in Figure~\ref{fig04}, and comparative computational results are recorded in
Table~\ref{table01} and Figure~\ref{fig06}(b).

\begin{figure}[htb]\centering
     \begin{minipage}[htb]{0.3\linewidth}\centering
        \subfigure[True model]		{\includegraphics[width=.98\linewidth]{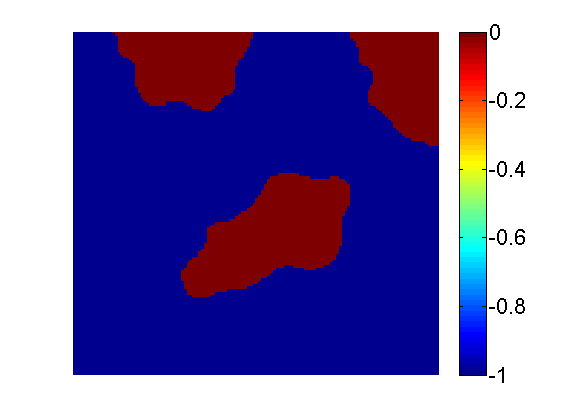}}
    \end{minipage}
    \begin{minipage}[htb]{0.3\linewidth}\centering
        \subfigure[Random Subset] {\includegraphics[width=.98\linewidth]{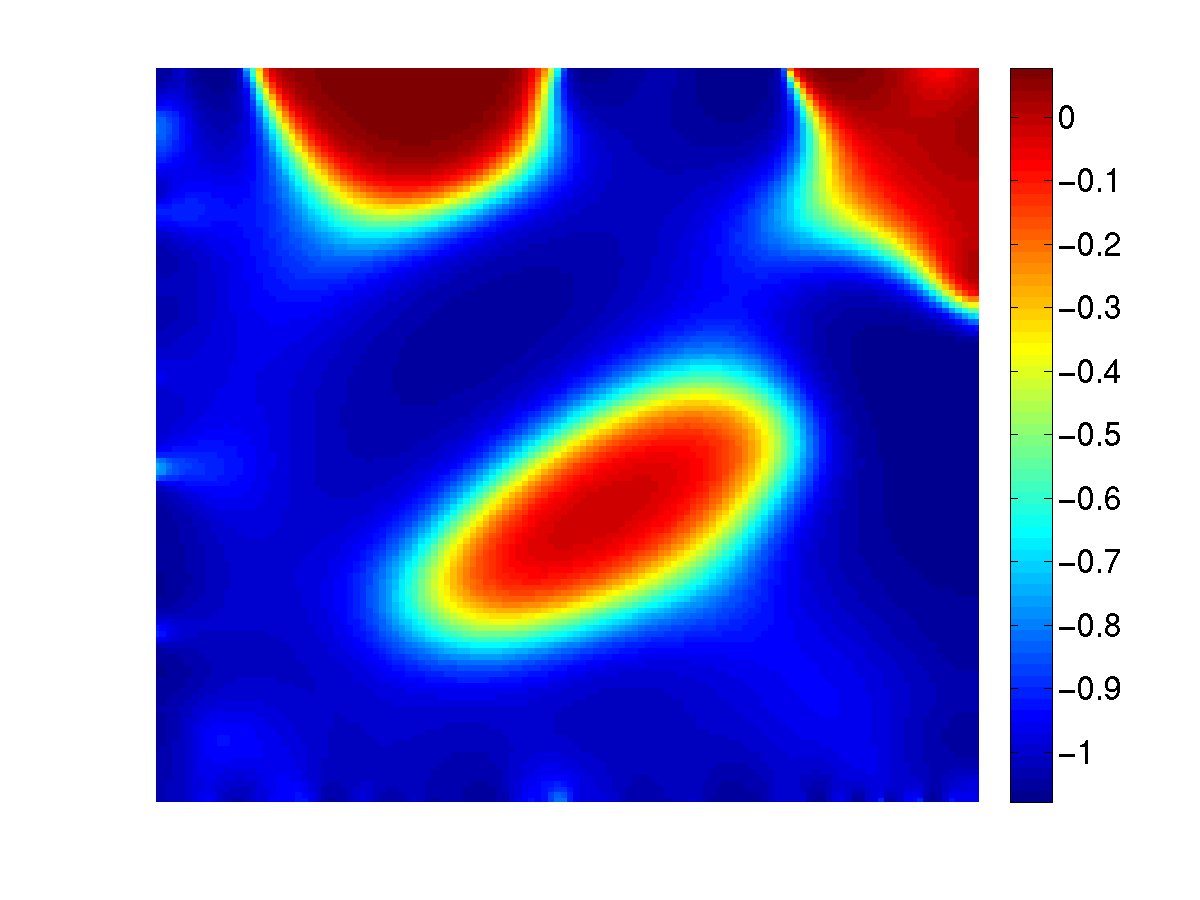}}
    \end{minipage}
    \begin{minipage}[htb]{0.3\linewidth}\centering
        \subfigure[Data Completion] {\includegraphics[width=.98\linewidth]{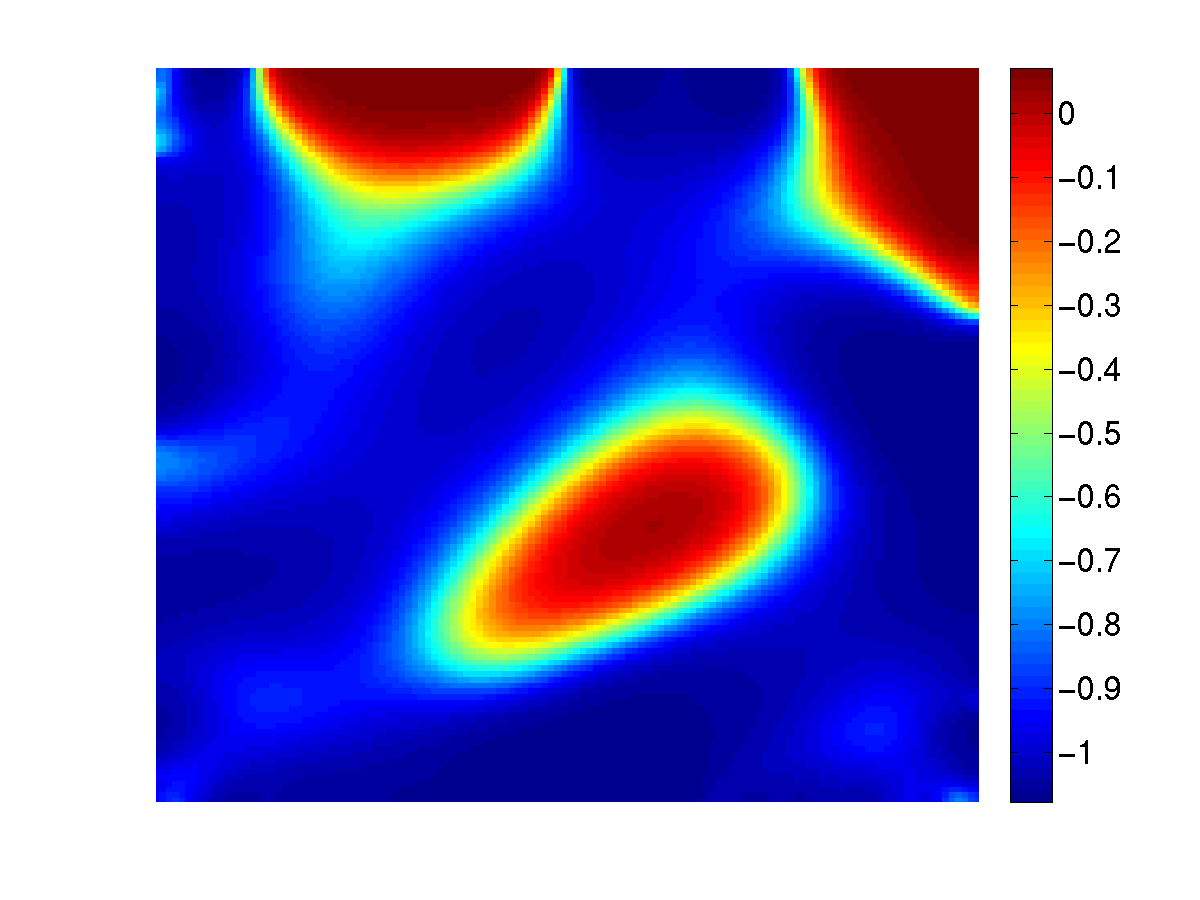}}
    \end{minipage}\hfill
    \caption{Example~\ref{exp02} -- reconstructed log conductivity with $50\%$ data missing and $5\%$ noise.  
		Regularization~\eqref{eqn_sub_02_datainterp} has been used to complete the data.
\label{fig04}}
\end{figure}
Similar observations to those in Example~\ref{exp01} generally apply here as well, despite the smaller
amount of original data.
\end{example}

\begin{example}
\label{exp03}
This is the same as Example~\ref{exp02} in terms of noise and the amount of missing data, except that the 
discontinuities in the conductivity are some distance away from the common measurement domain, 
so we use \eqref{eqn_sub_01_datainterp} to complete the data.
The same algorithm variants as in the previous two examples are compared, thus isolating the effect of
a smoother data approximant.

\begin{figure}[htb]\centering
     \begin{minipage}[htb]{0.3\linewidth}\centering
        \subfigure[True model]		{\includegraphics[width=.98\linewidth]{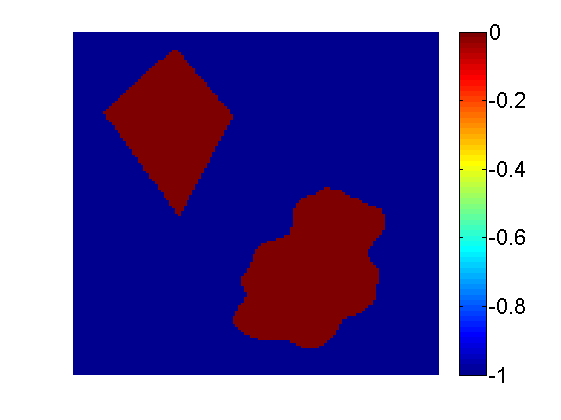}}
    \end{minipage}
    \begin{minipage}[htb]{0.3\linewidth}\centering
        \subfigure[Random Subset] {\includegraphics[width=.98\linewidth]{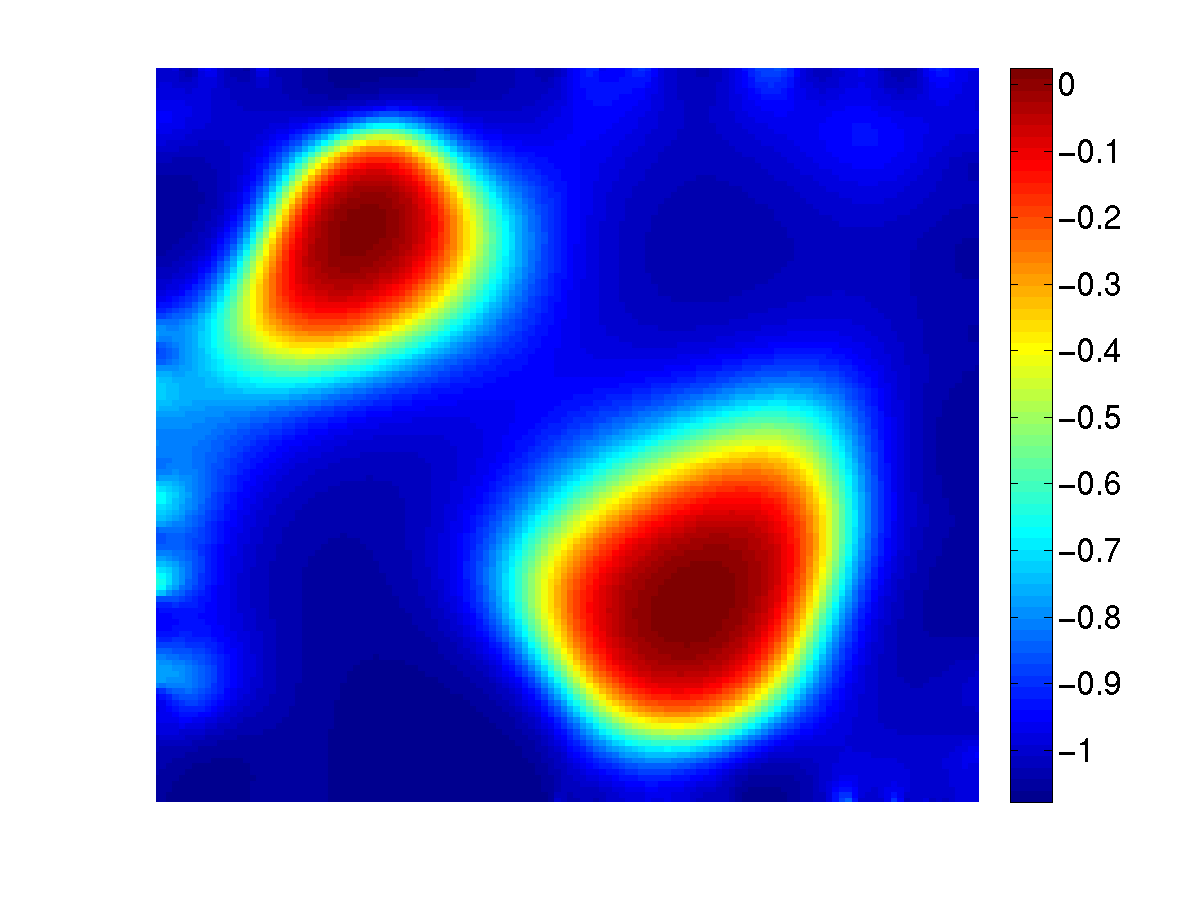}}
    \end{minipage}
    \begin{minipage}[htb]{0.3\linewidth}\centering
        \subfigure[Data Completion] {\includegraphics[width=.98\linewidth]{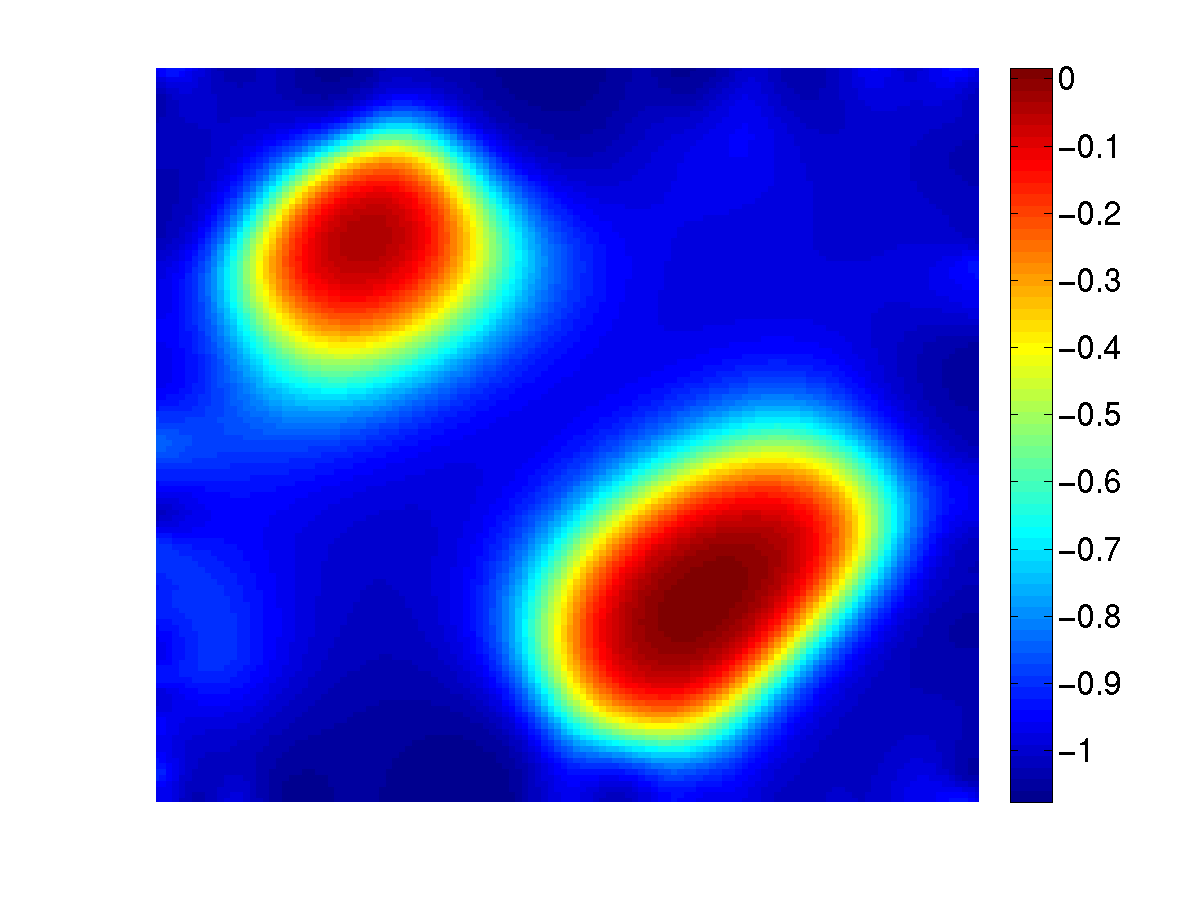}}
    \end{minipage}\hfill
    \caption{Example~\ref{exp03} -- reconstructed log conductivity with $50\%$ data missing and $5\%$ noise.
Regularization~\eqref{eqn_sub_01_datainterp} has been used to complete the data.
\label{fig05}}
\end{figure}

\begin{figure}[htb]
\centering
\subfigure[Example~\ref{exp01}]{
\includegraphics[scale=0.35]{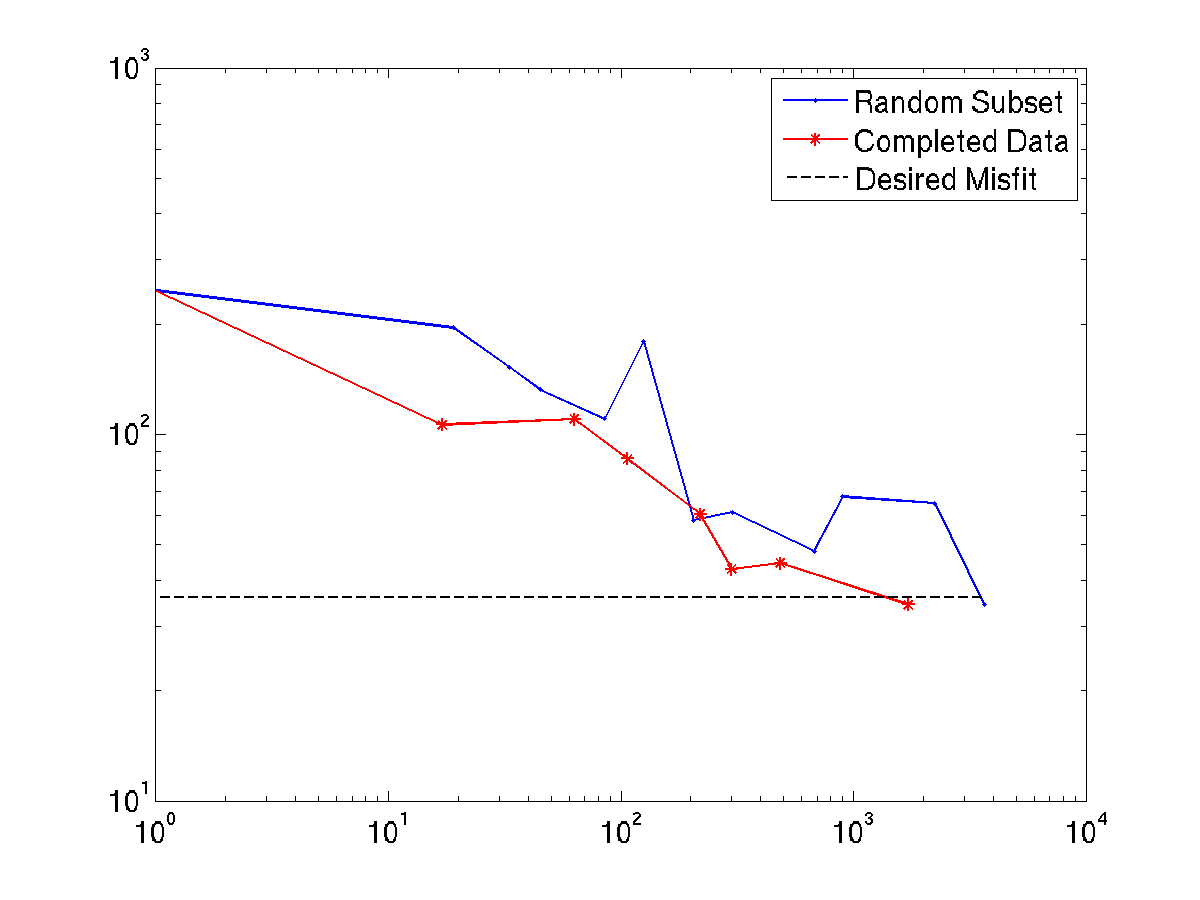}}
\subfigure[Example~\ref{exp02}]{
\includegraphics[scale=0.35]{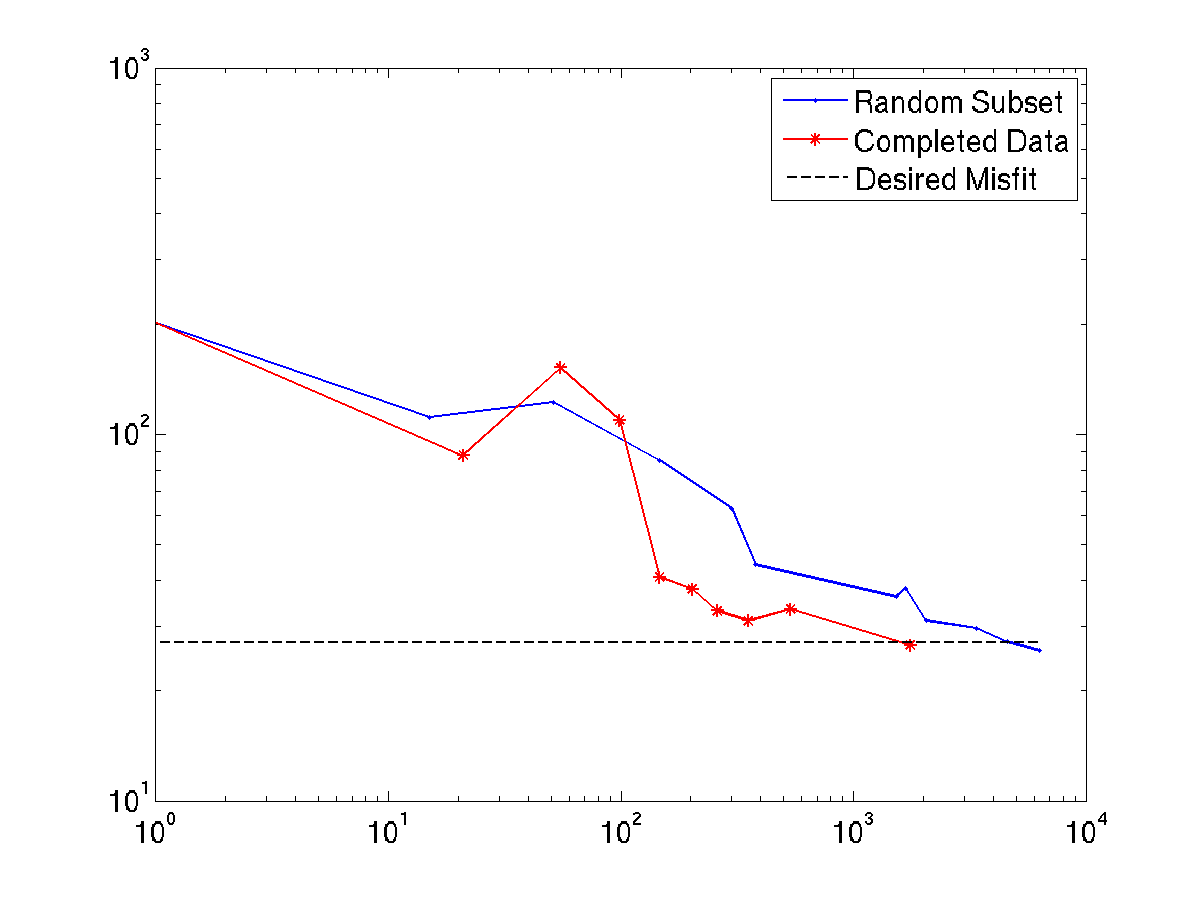}}
\subfigure[Example~\ref{exp03}]{
\includegraphics[scale=0.35]{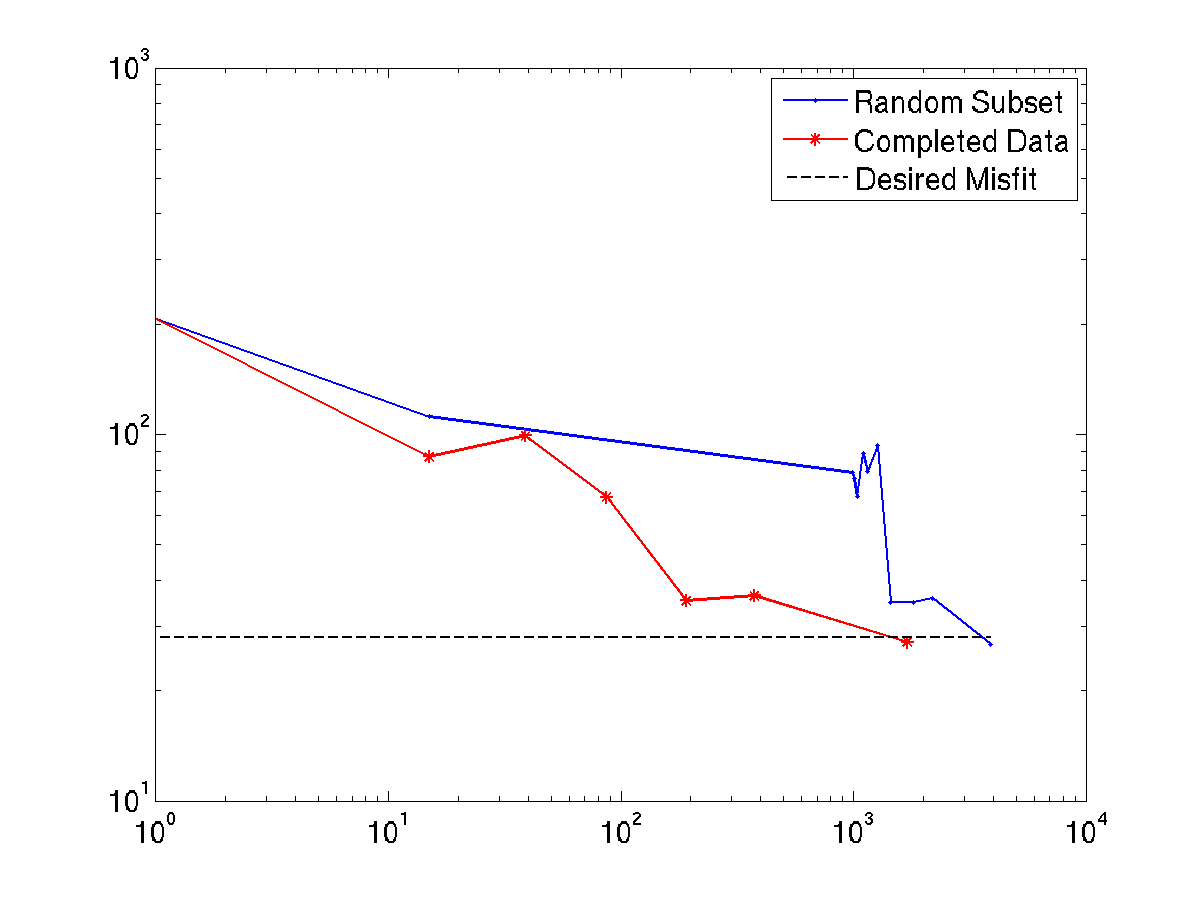}}
\caption{Data misfit vs. PDE count  for Examples~\hyperref[exp01]{1},~\hyperref[exp02]{2} and~\hyperref[exp03]{3}.
\label{fig06}}
\end{figure}
Results are recorded in Figure~\ref{fig05}, the third entry of Table~\ref{table01} and Figure~\ref{fig06}(c).
\end{example}

Figures~\ref{fig03},~\ref{fig04} and~\ref{fig05} in conjunction with Figure~\ref{fig06} as well as Table~\ref{table01}, 
reflect superiority of the SS method combined with data completion over the RS method with the original data. 
From the first three entries of Table~\ref{table01}, we see that the SS reconstruction with 
completed data can be done more efficiently by a factor of more than two. The quality of reconstruction is also very good. 
Note that the graph of the misfit for Data Completion lies mostly under that of Random Subset. 
This means that, given a fixed number of PDE solves, we obtain a lower (thus better) misfit for 
the former than for the latter. 

Next, we consider examples in 3D.

\begin{example}
\label{exp04}
In this example, the discontinuities in the true, piecewise constant conductivity extend all the way 
to the common measurement domain, see Figure~\ref{fig12}.
We therefore use~\eqref{eqn_sub_02_datainterp} to complete the data. 
The target object has the conductivity $\sigma_I = 1$ in a background with conductivity $\sigma_{II} = 0.1$. 
We add $2\%$ noise and knock out $50\%$ of the data. 
Furthermore, we consider the relaxed stopping criterion~\eqref{alg1_stop_b}.
With the original data (hence using RS), the variant (\ref{alg1_data_i},~\ref{alg1_stop_b}) is employed, 
and this is compared against the variant (\ref{alg1_data_ii},~\ref{alg1_stop_b}) with SS applied to the completed data. 
For the latter case, the stopping tolerance is adjusted as discussed in Section~\ref{sec:comp_alg}. 

\begin{figure}[htb]
\centering
\mbox{
\includegraphics[scale=0.3]{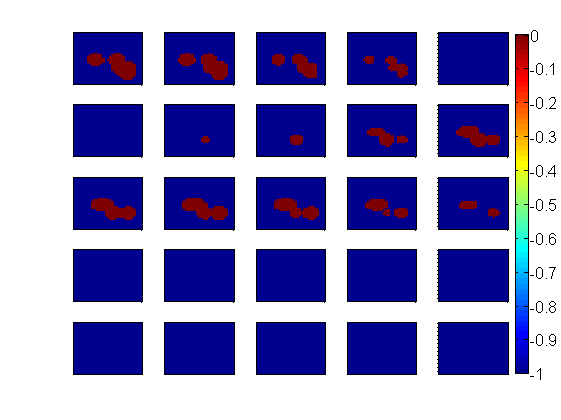}
\includegraphics[scale=0.3]{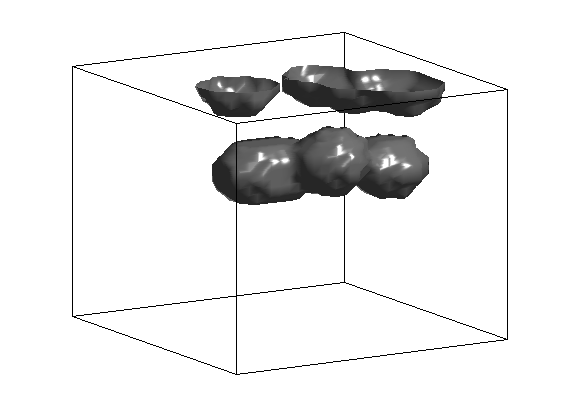}}
\caption{True Model for Example~\ref{exp04}.\label{fig12}}
\end{figure}

\begin{figure*}[htb]\centering
    \begin{minipage}[htb]{0.25\linewidth}\centering
        \subfigure[RS slices]
        {\includegraphics[width=.98\linewidth]{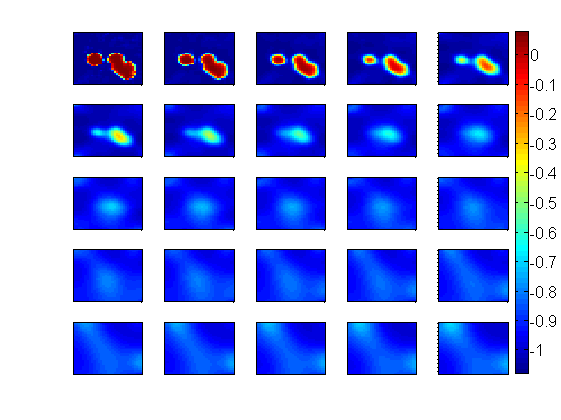}}
    \end{minipage}
    \begin{minipage}[htb]{0.2\linewidth}\centering
        \subfigure[3D view]        {\includegraphics[width=.98\linewidth]{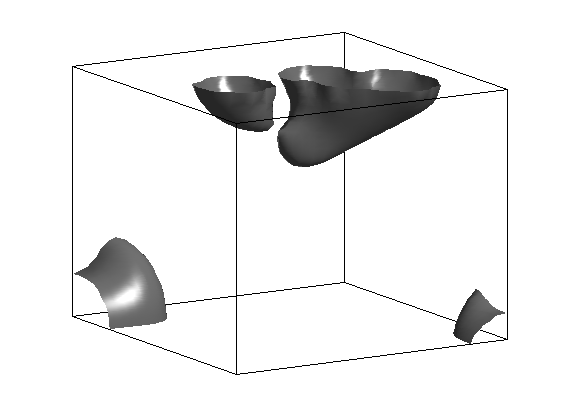}}
    \end{minipage}\hfill
    \begin{minipage}[htb]{0.25\linewidth}\centering
        \subfigure[DC slices]
        {\includegraphics[width=.98\linewidth]{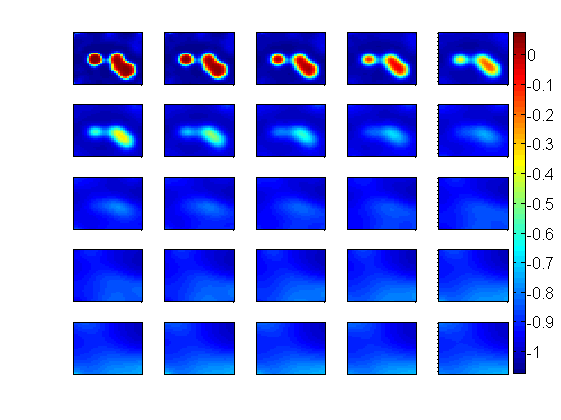}}
    \end{minipage}\hfill
    \begin{minipage}[htb]{0.2\linewidth}\centering
        \subfigure[3D view]
        {\includegraphics[width=.98\linewidth]{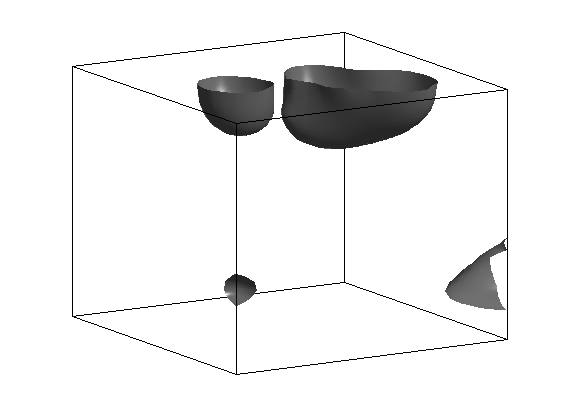}}
    \end{minipage}\hfill
    \caption{Example~\ref{exp04} -- reconstructed log conductivity for the 3D model with (a,b) Random Subset, 
		(c,d) Data Completion for the case of $2\%$ noise and $50\%$ of data missing. 
		Regularization~\eqref{eqn_sub_02_datainterp} has been used to complete the data.
\label{fig13}}
\end{figure*}
Reconstruction results are depicted in Figure~\ref{fig13}, 
and work estimates are gathered in the $4$th entry of Table~\ref{table01}.
It can be seen that the results using data completion, obtained at about $1/7$th the cost, are comparable to those
obtained with RS applied to the original data.
\end{example}

\begin{example}
\label{exp05}

The underlying model in this example is the same as that in Example~\ref{exp04} except that, 
since we intend to plot the misfit on the entire data at every GN iteration, 
we decrease the reconstruction mesh resolution to $17^3$.
Also, $30\%$ of the data requires completion, and we use the level set transfer function~\eqref{2.7} 
to reconstruct the model. With the original data, we use the variant (\ref{alg1_data_i},~\ref{alg1_stop_a}), 
while the variant (\ref{alg1_data_iii},~\ref{alg1_stop_a}) is used with the completed data. 
The reconstruction results are recorded in Figure~\ref{fig15}, and performance indicators appear
in Figure~\ref{fig16} as well as Table~\ref{table01}.

\begin{figure*}[htb]\centering    \begin{minipage}[htb]{0.25\linewidth}\centering
        \subfigure[RS slices]
        {\includegraphics[width=.98\linewidth]{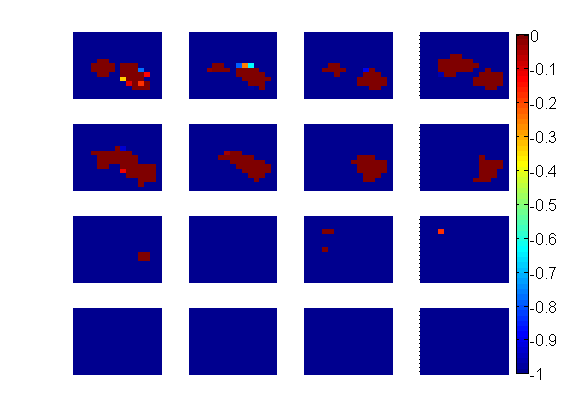}}
    \end{minipage}
    \begin{minipage}[htb]{0.2\linewidth}\centering
        \subfigure[3D view]
        {\includegraphics[width=.98\linewidth]{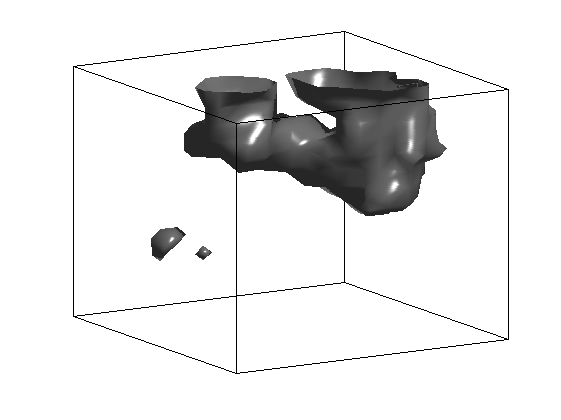}}
    \end{minipage}\hfill
    \begin{minipage}[htb]{0.25\linewidth}\centering
        \subfigure[DC slices]
        {\includegraphics[width=.98\linewidth]{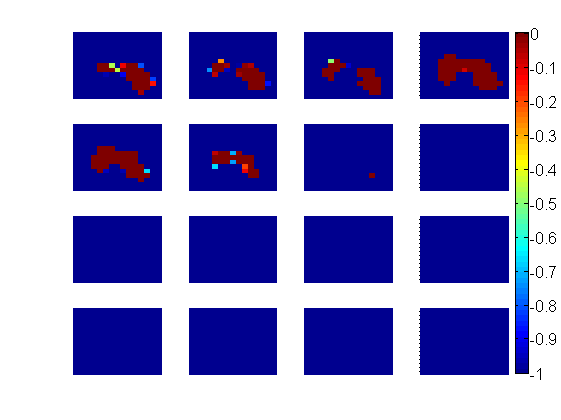}}
    \end{minipage}\hfill
    \begin{minipage}[htb]{0.2\linewidth}\centering
        \subfigure[3D view]
        {\includegraphics[width=.98\linewidth]{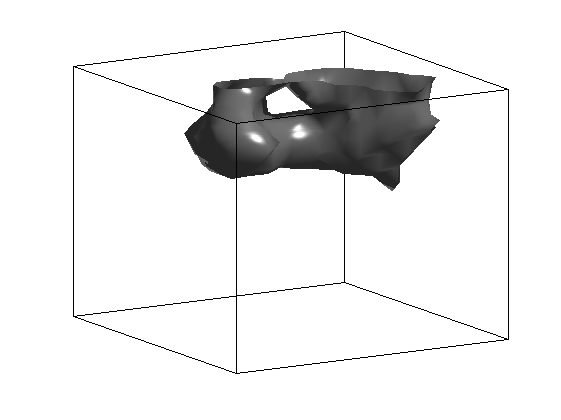}}
    \end{minipage}\hfill
    \caption{Example~\ref{exp05} -- reconstructed log conductivity for the 3D model using 
the level set method with (a,b) Random Subset, 
(c,d) Data Completion for the case of $2\%$ noise and $30\%$ of data missing. 
Regularization~\eqref{eqn_sub_02_datainterp} 
		has been used to complete the data.
\label{fig15}}
\end{figure*}

\begin{figure}[[!ht]
\centering
\includegraphics[scale=0.35]{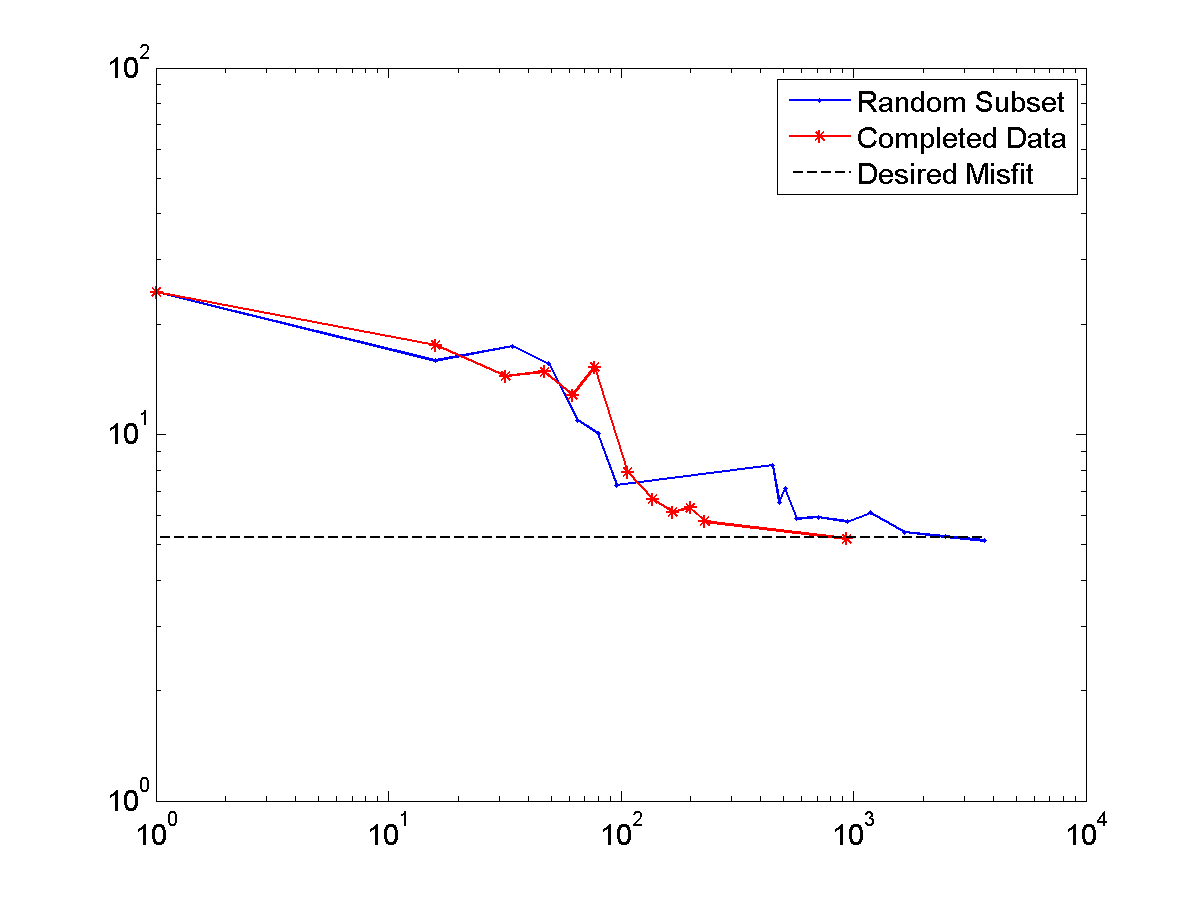}
\caption{Data misfit vs. PDE count  for Example~\ref{exp05}.
\label{fig16}}
\end{figure}

The algorithm proposed here produces a better reconstruction than RS on the original data.
A relative efficiency observation can be made from Table~\ref{table01}, where a factor of roughly $4$ is revealed.
\end{example}

\begin{example}
\label{exp06}
This is exactly the same as Example~\ref{exp04}, except that we use the level set transfer function~\eqref{2.7} 
to reconstruct the model. 
The same variants of Algorithm~\ref{alg1} as in Example~\ref{exp04} are employed.

\begin{figure*}[htb]\centering
    \begin{minipage}[htb]{0.25\linewidth}\centering
        \subfigure[RS slices]
        {\includegraphics[width=.98\linewidth]{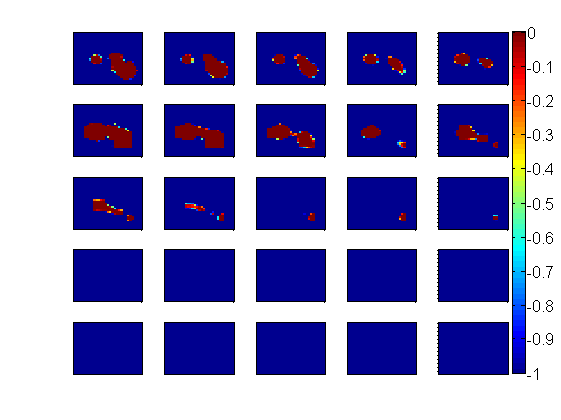}}
    \end{minipage}
    \begin{minipage}[htb]{0.2\linewidth}\centering
        \subfigure[3D view]        {\includegraphics[width=.98\linewidth]{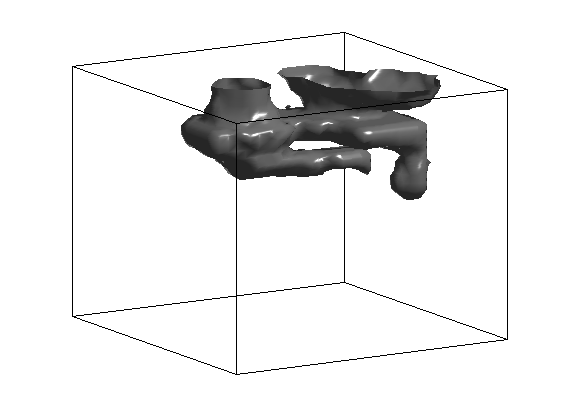}}
    \end{minipage}\hfill
    \begin{minipage}[htb]{0.25\linewidth}\centering
        \subfigure[DC slices]
        {\includegraphics[width=.98\linewidth]{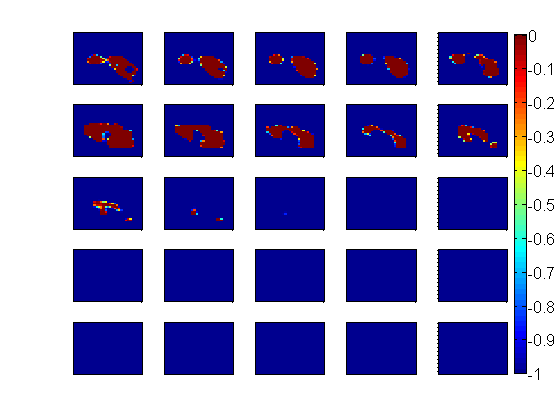}}
    \end{minipage}\hfill
    \begin{minipage}[htb]{0.2\linewidth}\centering
        \subfigure[3D view]
 {\includegraphics[width=.98\linewidth]{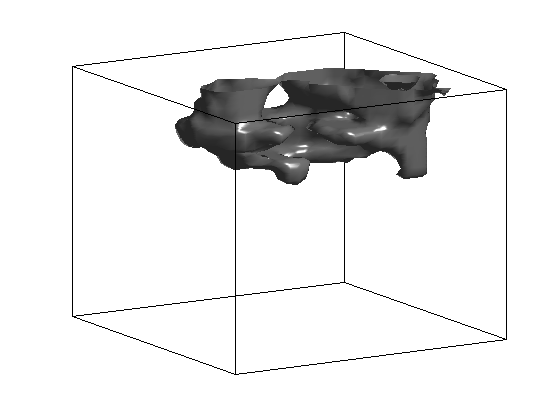}}
    \end{minipage}\hfill
    \caption{Example~\ref{exp06} -- reconstructed log conductivity for the 3D model using 
the level set method with (a,b) Random Subset, 
		(c,d) Data Completion for the case of $2\%$ noise and $50\%$ of data missing. 
		Regularization~\eqref{eqn_sub_02_datainterp} has been used to complete the data.
\label{fig17}}
\end{figure*}

It is evident from Figure~\ref{fig17} that employing the level set formulation allows a significantly 
better quality reconstruction than in Example~\ref{exp04}. 
This is expected, as much stronger assumptions on the true model are utilized. 
It was shown in~\cite{doas,rodoas} that using level set functions can greatly reduce 
the total amount of work, and this is observed here as well.

Whereas in all previous examples convergence of the modified GN iterations from a zero initial
guess was fast and uneventful, typically requiring fewer than 10 iterations,
the level set result of this example depends on $\mm_0$ in a more erratic manner. 
This reflects the underlying uncertainty of the inversion, with the initial guess
$\mm_0$ playing the role of a prior.

\end{example}

 It can be clearly seen from the results of Examples~\ref{exp04},~\ref{exp05} and~\ref{exp06} 
 that Algorithm~\ref{alg1} does a great job recovering the model using the completed data plus the SS method
as compared to RS with the original data. This is so both in terms of total work and the quality of the recovered model.
Note that for all reconstructions, the conductive object placed deeper than the ones closer 
to the surface is not recovered well. 
This is due to the fact that we only measure on the surface and the information coming from this deep conductive 
object is majorized by that coming from the objects closer to the surface. 

\begin{example}
\label{exp07}
In this 3D example, we examine the performance of our data completion approach for more severe cases of missing data. 
For this example, we place a target object of conductivity $\sigma_I = 1$ in a background with conductivity $\sigma_{II} = 0.1$, 
see Figure~\ref{fig18}, and $2\%$ noise is added to the ``exact'' data. 
Then we knock out $70\%$ of the data and use \eqref{eqn_sub_01_datainterp} to complete it. 
The algorithm variants employed are the same as in Examples~\ref{exp04} and ~\ref{exp06}.

Results are gathered in Figures~\ref{fig19} as well as Table~\ref{table01}.
The data completion plus simultaneous sources algorithm again does well, with an efficiency factor $\approx 4$.
\end{example}

\begin{figure}[htb]
\centering
\mbox{
\includegraphics[scale=0.3]{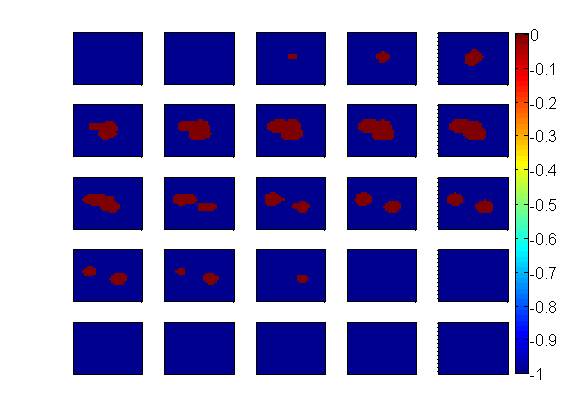}
\includegraphics[scale=0.3]{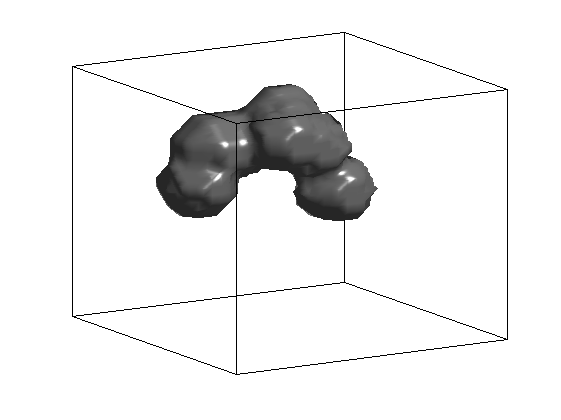}}
\caption{True Model for Example~\ref{exp07}.\label{fig18}}
\end{figure}

\begin{figure*}[htb]\centering
    \begin{minipage}[htb]{0.25\linewidth}\centering
        \subfigure[RS slices]
        {\includegraphics[width=.98\linewidth]{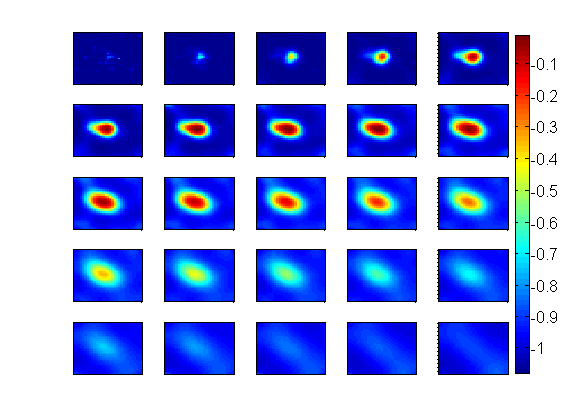}}
    \end{minipage}
    \begin{minipage}[htb]{0.2\linewidth}\centering
        \subfigure[3D view]
       {\includegraphics[width=.98\linewidth]{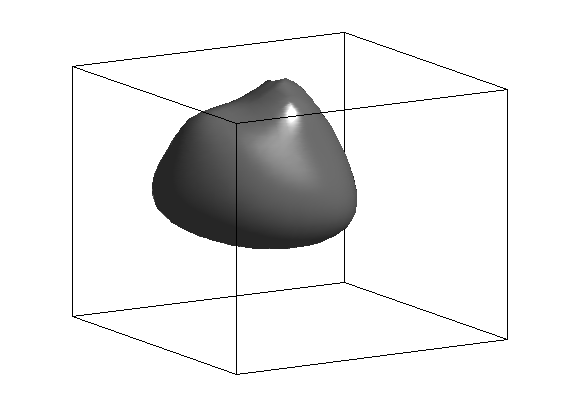}}
    \end{minipage}\hfill
    \begin{minipage}[htb]{0.25\linewidth}\centering
        \subfigure[DC slices]
        {\includegraphics[width=.98\linewidth]{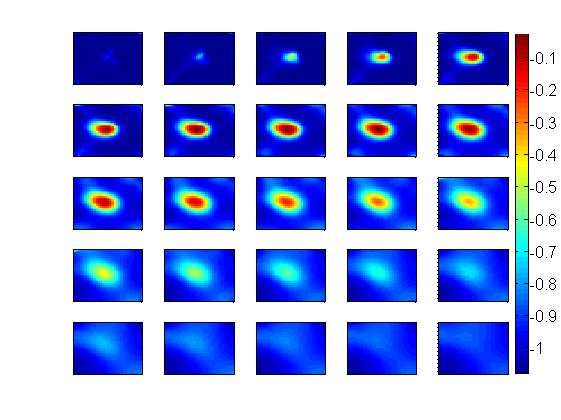}}
    \end{minipage}\hfill
    \begin{minipage}[htb]{0.2\linewidth}\centering
        \subfigure[3D view]
        {\includegraphics[width=.98\linewidth]{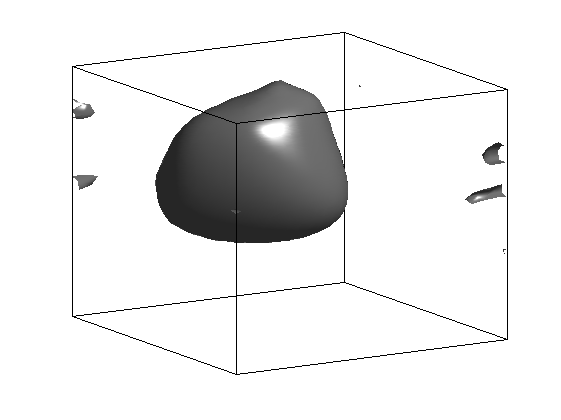}}
    \end{minipage}\hfill
    \caption{Example~\ref{exp07} -- reconstructed log conductivity for the 3D model with (a,b) Random Subset, 
		(c,d) Data Completion for the case of $2\%$ noise and $70\%$ data missing. Regularization~\eqref{eqn_sub_01_datainterp} has been used to complete the data.
\label{fig19}}
\end{figure*}

\section{Conclusions and further comments}
\label{sec:conclusions}

This paper is a sequel to~\cite{rodoas} in which we studied the case where sources share the same receivers.
Here we have focused on the very practical case where sources do not share the same receivers 
yet are distributed in a particular manner,
and have proposed a new approach based on 
appropriately regularized data completion.
Our data completion methods are motivated by theory in Sobolev spaces regarding the properties 
of weak solutions along the domain boundary. 
The resulting completed data allows an efficient use of the methods developed in~\cite{rodoas}
as well as utilization of a relaxed stopping criterion.
Our approach shows great success in cases of moderate data completion, say up to 60-70\%. 
In such cases we have demonstrated that, utilizing some variant of Algorithm~\ref{alg1}, an execution speedup factor of at least 2 
and often much more can be achieved while obtaining excellent reconstructions.

It needs to be emphasized that a blind employment of some interpolation/approximation method 
would not take into account available a priori information about the sought signal.
In contrast, the method developed in this paper, while being very simple, 
is in fact built upon such a priori information, and is theoretically justified.

Note that with the methods of Section~\ref{sec:datainterp} we have also replaced the
original data with new, approximate data. Alternatively we could
keep the original data, and just add the missing data sampled from $v_i$
at appropriate locations. 
The potential advantage of doing this is that fewer changes are made to the original problem,
so it would seem plausible that the data extension will produce results
that are close to the more expensive inversion without using the simultanous
sources method, at least when there are only a few missing receivers. 
However, we found in practice that this method
yields similar or worse reconstructions for moderate or large amounts of missing
data as compared to the methods of Section~\ref{sec:datainterp}.

For severe cases of missing data, say $80\%$ or more, 
we do not recommend data completion in the present context
as a safe approach. With so much completion the bias in the completed field could 
overwhelm the given observed data, and the recovered model may not be correct. 
In such cases, one can use the RS method applied to the original data.
A good initial guess for this method may still be obtained with the SS
method applied to the completed data.
Thus, one can always start with the most daring variant $(ii,b)$ of Algorithm~\ref{alg1},
and add a more conservative run of variant $(i,b)$ on top if necessary. 

If the forward problem is very diffusive and has a strong smoothing effect, 
as is the case for the DC-resistivity and EIT
problems, then data completion can be attempted using a (hopefully) good guess of 
the sought model $\mm$ by solving the forward problem
and evaluating the solution wherever necessary~\cite{hach12}.
The rationale here is that even relatively large changes in $m(\xx)$ produce only small 
changes in the fields $u_i(\xx)$.
However, such a prior might prove dominant, hence risky,
and the data produced in this way, unlike the original data, 
no longer have natural high frequency noise components.
Indeed, a potential advantage of this approach is in using the difference between 
the original measured data and 
the calculated prior field at the same locations for
estimating the noise level $\epsilon$
for a subsequent application of the Morozov discrepancy principle~\cite{vogelbook,ehn1}.

In this paper we have focused on data completion, 
using whenever possible the same computational setting as in~\cite{rodoas}, which is our base reference. 
Other approaches to reduce the overall computational costs are certainly possible. 
These include adapting the number of inner PCG iterations in the modified GN outer iteration 
(see~\cite{doas3}) and adaptive gridding for $\mm (\xx)$ 
(see, e.g., \cite{haheas} and references therein). 
Such techniques are essentially independent of the focus here. 
At the same time, they can be incorporated or fused together with our stochastic algorithms,
further improving efficiency: 
effective ways for doing this form a topic for future research.   

The specific data completion techniques proposed in this paper have been justified and used 
in our model DC resistivity problem. 
However, the overall idea can be extended to other PDE based inverse problems as well 
by studying the properties of the solution of the forward problem. 
One first needs to see what the PDE solutions are expected to behave like on the measurement domain, 
for example on a portion of the boundary, and then imposing this prior 
knowledge in the form of an appropriate regularizer in the data completion formulation. 
Following that, the rest can be similar to our approach here. 
Investigating such extensions to other PDE models is a subject for future studies.

\section*{Acknowledgments}
\label{acknowledgment}
The authors would like to thank Drs. Adriano De Cezaro and Eldad Haber for several fruitful discussions. 



\end{document}